\documentclass[11pt]{amsart}
\usepackage{amsthm,amsmath,amssymb,amscd,epsfig,epic,eepic,graphicx,color}

{\end{list}}
{
   \newtheorem{theorem}{Theorem}[section]
   \newtheorem{proposition}[theorem]{Proposition}
   \newtheorem{lemma}[theorem]{Lemma}

   \newtheorem{corollary}[theorem]{Corollary}
   \newtheorem{conjecture}[theorem]{Conjecture}

}
{\theoremstyle{definition}

}
{\theoremstyle{remark}

}

\newcommand{\RR}{{\mathbb{R}}}

\newcommand{\ZZ}{{\mathbb{Z}}}

\newcommand{\go}{\Rightarrow}
\newcommand{\longgo}{\Longrightarrow}

\begin{document}

\title{On Oda's Strong Factorization Conjecture}                                          

\author{Sergio Da Silva}
\address{Mathematics Department\\ University of Toronto, \\
40 St. George Street\\
Toronto, Ontario, Canada M5S 2E4}
\email{sergio.dasilva@utoronto.ca}
\author{Kalle Karu}
\address{Mathematics Department\\ University of British Columbia \\
  1984 Mathematics Road\\
Vancouver, B.C. Canada V6T 1Z2}
\email{karu@math.ubc.ca}
\thanks{This work was partially supported by NSERC USRA and Discovery grants.}

\begin{abstract}
The Oda's Strong Factorization Conjecture states that a proper birational map
between smooth toric varieties can be decomposed as a sequence of
smooth toric blowups followed by a sequence of smooth toric
blowdowns. This article describes an algorithm that conjecturally
constructs such a decomposition. Several reductions and
simplifications of the algorithm are presented and some special cases
of the conjecture are proved.  
\end{abstract}

\maketitle

\section{Introduction}

The general strong factorization problem asks if a proper birational
map between nonsingular varieties (in characteristic zero) can be
factored into a sequence of blowups with nonsingular centers followed
by a sequence of inverses of such maps. Oda \cite{Miyake-Oda} posed
the same problem for toric varieties and toric birational maps. Since
toric varieties are defined by combinatorial data, the conjecture for
toric varieties also takes a combinatorial form.

A nonsingular toric variety is determined by a nonsingular fan and a smooth toric blowup corresponds to a smooth star subdivision of the fan. The conjecture then is:

\begin{conjecture}[Oda]\label{conj-oda} Given two nonsingular fans $\Delta_1$ and $\Delta_2$ with the same support, there exists a third fan $\Delta_3$  that can be reached from both $\Delta_1$ and $\Delta_2$ by sequences of smooth star subdivisions.
\end{conjecture}

As the terminology suggests, there also exists a weak version of the factorization problem in which blowups and blowdowns are allowed in any order. This weak conjecture is known to hold for toric varieties \cite{Wlodarczyk1, Morelli1, Abramovich-Matsuki-Rashid} and also for general varieties in characteristic zero \cite{Wlodarczyk3, AKMW}. The strong factorization conjecture is open in all cases in dimension $3$ or higher.

In this article we study a simple algorithm that is conjectured to
construct the strong factorization for toric varieties. The problem is
that the algorithm may run into an infinite loop and never
finish. However, computer experiments and proofs of several special
cases suggest that the algorithm is always finite and solves
Conjecture~\ref{conj-oda}. 

To construct a strong factorization between two fans, we start with a
weak factorization that is known to exist. In other words, we assume
that we can get from one fan to the other by a sequence of smooth star
subdivisions and smooth star assemblies (the inverses of star
subdivisions). The goal is to ``commute'' star assemblies and star
subdivisions to have all assemblies after the subdivisions. The algorithm takes one star assembly followed by one star subdivision and replaces the pair with a sequence of star subdivisions followed by a sequence of star assemblies. This step is repeated until all star subdivisions precede the star assemblies. 

Let us consider $3$-dimensional nonsingular fans. We draw such a fan as a cross section, which is a $2$-dimensional simplicial complex. We may assume that all maximal cones have dimension $3$ and that all star subdivisions have their subdivision rays in the middle of $2$-dimensional cones. In coordinates, a cone generated by $v_1, v_2, v_3$ is divided by a star subdivision into two cones generated by   $v_1+v_2, v_2, v_3$ and   $v_1, v_1+ v_2, v_3$; the ray generated by $v_1+v_2$ is called the subdivision ray.

Now given two star subdivisions of one fan, we need to construct a
common refinement by further star subdividing the two new
fans. Figure~\ref{fig-alg} shows two ways of doing this, denoted by
$A$ and $B$. In this figure, the fan we start with consists of a
single cone generated by $v_1, v_2, v_3$. The two subdivisions have
subdivision rays generated by $v_1+v_2$ and $v_1+v_3$. Both factorizations $A$ and $B$ replace the assembly-subdivision pair by two subdivisions followed by two assemblies. (We read the sequence of maps from left to right, starting from the fan $\Delta_1$ and ending with $\Delta_2$.) 

\begin{figure}[h] \label{fig-alg}
\centerline{\setlength{\unitlength}{0.00029167in}
\begingroup\makeatletter\ifx\SetFigFontNFSS\undefined%
\gdef\SetFigFontNFSS#1#2#3#4#5{%
  \reset@font\fontsize{#1}{#2pt}%
  \fontfamily{#3}\fontseries{#4}\fontshape{#5}%
  \selectfont}%
\fi\endgroup%
{\renewcommand{\dashlinestretch}{30}
\begin{picture}(17427,7392)(0,-10)
\path(3315,2265)(3915,3465)(4515,2265)(3315,2265)
\path(4215,2865)(3315,2265)
\path(5265,2115)(6315,2115)
\path(5265,1815)(6315,1815)
\path(6165,2415)(6615,1965)(6165,1515)
\path(10665,2265)(11265,3465)(11865,2265)(10665,2265)
\path(11588,2830)(10688,2230)
\path(8265,4215)(8865,5415)(9465,4215)(8265,4215)
\path(7665,2265)(8265,3465)(8865,2265)(7665,2265)
\path(7965,2865)(8865,2265)
\path(10065,4215)(10665,5415)(11265,4215)(10065,4215)
\path(9165,6165)(9765,7365)(10365,6165)(9165,6165)
\path(9465,4215)(8565,4815)(9165,4815)
\path(10065,4215)(10965,4815)(10365,4815)
\path(10365,6165)(9465,6765)(10065,6765)(9165,6165)
\path(9165,6165)(9765,7365)(10365,6165)(9165,6165)
\path(13215,2265)(13815,3465)(14415,2265)(13215,2265)
\path(13815,4215)(14415,5415)(15015,4215)(13815,4215)
\path(14715,6165)(15315,7365)(15915,6165)(14715,6165)
\path(16215,2265)(16815,3465)(17415,2265)(16215,2265)
\path(13515,2865)(14415,2265)
\path(17115,2865)(16215,2265)
\path(15015,4215)(14115,4815)
\path(14415,5415)(14415,4665)(13815,4215)
\path(15615,4215)(16215,5415)(16815,4215)(15615,4215)
\path(15615,4215)(16515,4815)
\path(16215,5415)(16215,4665)(16815,4215)
\path(14715,6165)(15615,6765)
\path(15015,6765)(15915,6165)
\path(15315,7365)(15315,6615)
\path(915,2265)(1515,3465)(2115,2265)(915,2265)
\path(1215,2865)(2115,2265)
\path(2115,165)(2715,1365)(3315,165)(2115,165)
\path(1515,1965)(1965,1215)
\path(1877.536,1302.464)(1965.000,1215.000)(1928.985,1333.334)
\path(3915,1965)(3315,1215)
\path(3366.537,1327.445)(3315.000,1215.000)(3413.389,1289.963)
\path(8865,3915)(8715,3315)
\path(8715.000,3438.693)(8715.000,3315.000)(8773.209,3424.141)
\path(9465,5865)(9315,5265)
\path(9315.000,5388.693)(9315.000,5265.000)(9373.209,5374.141)
\path(10065,5865)(10215,5265)
\path(10156.791,5374.141)(10215.000,5265.000)(10215.000,5388.693)
\path(10665,3915)(10815,3315)
\path(10756.791,3424.141)(10815.000,3315.000)(10815.000,3438.693)
\path(14415,3915)(14265,3315)
\path(14265.000,3438.693)(14265.000,3315.000)(14323.209,3424.141)
\path(16215,3915)(16365,3315)
\path(16306.791,3424.141)(16365.000,3315.000)(16365.000,3438.693)
\path(15615,5865)(15765,5265)
\path(15706.791,5374.141)(15765.000,5265.000)(15765.000,5388.693)
\path(15050,5856)(14900,5256)
\path(14900.000,5379.693)(14900.000,5256.000)(14958.209,5365.141)
\put(4365,2865){\makebox(0,0)[lb]{\smash{{\SetFigFontNFSS{5}{6.0}{\rmdefault}{\mddefault}{\updefault}V1+V3}}}}
\put(3465,15){\makebox(0,0)[lb]{\smash{{\SetFigFontNFSS{5}{6.0}{\rmdefault}{\mddefault}{\updefault}V3}}}}
\put(1665,15){\makebox(0,0)[lb]{\smash{{\SetFigFontNFSS{5}{6.0}{\rmdefault}{\mddefault}{\updefault}V2}}}}
\put(2565,1515){\makebox(0,0)[lb]{\smash{{\SetFigFontNFSS{5}{6.0}{\rmdefault}{\mddefault}{\updefault}V1}}}}
\put(15,2940){\makebox(0,0)[lb]{\smash{{\SetFigFontNFSS{5}{6.0}{\rmdefault}{\mddefault}{\updefault}V1+V2}}}}
\put(9465,765){\makebox(0,0)[lb]{\smash{{\SetFigFontNFSS{11}{13.2}{\rmdefault}{\mddefault}{\updefault}A}}}}
\put(15315,765){\makebox(0,0)[lb]{\smash{{\SetFigFontNFSS{11}{13.2}{\rmdefault}{\mddefault}{\updefault}B}}}}
\end{picture}
}}
\caption{Algorithms $A$ and $B$.}
\end{figure}
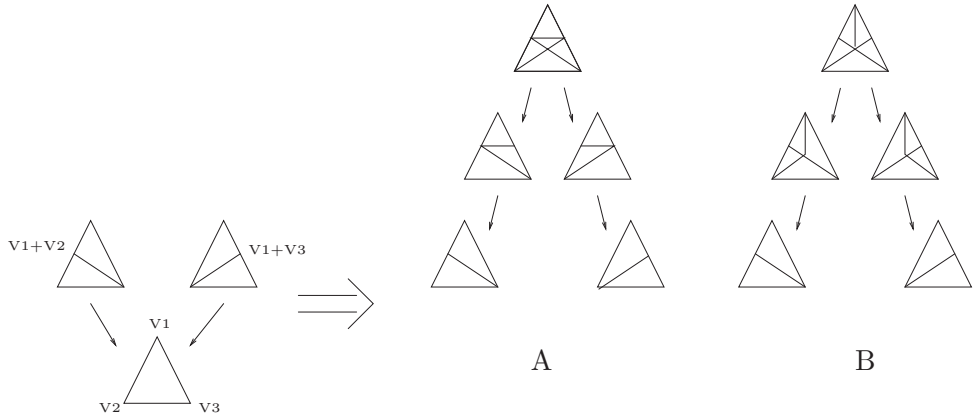  

Figure~\ref{fig-alg}  describes two factorization algorithms $A$ and $B$ completely. If we have, instead of a single cone as in the picture, a global fan and its two star subdivisions, then the subdivisions commute if the subdivision rays do not lie in one cone. If the subdivision rays do lie in the same cone, then Figure~\ref{fig-alg} tells us how to factor the diagram. (If the cone containing the two subdivision rays lies in a bigger fan, the star subdivisions shown in   Figure~\ref{fig-alg} can clearly be exteded to star subdivisons of the bigger fan.)

Figure~\ref{fig-example}(a) shows algorithm $A$ applied to factor two star assemblies and one star subdivision (the lower edge of the diagram) into four star subdivisions followed by five star assemblies (the top of the diagram). 

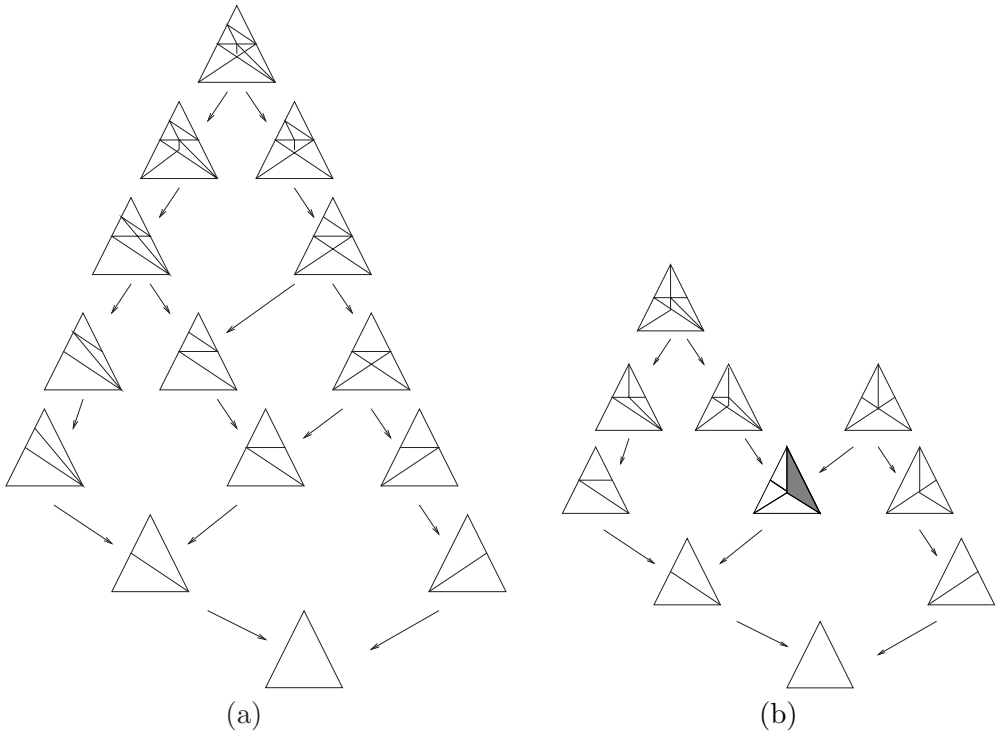
\begin{figure}[h] 
 \begin{minipage}[b]{2.5in} 
\centering
\setlength{\unitlength}{0.00033333in}
\begingroup\makeatletter\ifx\SetFigFontNFSS\undefined%
\gdef\SetFigFontNFSS#1#2#3#4#5{%
  \reset@font\fontsize{#1}{#2pt}%
  \fontfamily{#3}\fontseries{#4}\fontshape{#5}%
  \selectfont}%
\fi\endgroup%
{\renewcommand{\dashlinestretch}{30}
\begin{picture}(7824,10689)(0,-10)
\path(4062,12)(4662,1212)(5262,12)
	(4062,12)(4062,12)
\path(2412,4662)(3012,5862)(3612,4662)
	(2412,4662)(2412,4662)
\path(1362,6462)(1962,7662)(2562,6462)
	(1362,6462)(1362,6462)
\path(4512,6462)(5112,7662)(5712,6462)
	(4512,6462)(4512,6462)
\path(2112,7962)(2712,9162)(3312,7962)
	(2112,7962)(2112,7962)
\path(3012,9462)(3612,10662)(4212,9462)
	(3012,9462)(3012,9462)
\path(5112,4662)(5712,5862)(6312,4662)
	(5112,4662)(5112,4662)
\path(5862,3162)(6462,4362)(7062,3162)
	(5862,3162)(5862,3162)
\path(6612,1512)(7212,2712)(7812,1512)
	(6612,1512)(6612,1512)
\path(3912,7962)(4512,9162)(5112,7962)
	(3912,7962)(3912,7962)
\path(12,3162)(612,4362)(1212,3162)
	(12,3162)(12,3162)
\path(612,4662)(1212,5862)(1812,4662)
	(612,4662)(612,4662)
\path(1662,1512)(2262,2712)(2862,1512)
	(1662,1512)(1662,1512)
\path(3462,3162)(4062,4362)(4662,3162)
	(3462,3162)(3462,3162)
\path(1962,2112)(2862,1512)
\path(312,3762)(1212,3162)(462,4062)
\path(7512,2112)(6612,1512)
\path(4662,3162)(3762,3762)(4362,3762)
\path(5862,3162)(6762,3762)(6162,3762)
\path(5112,4662)(6012,5262)(5412,5262)(6312,4662)
\path(3612,4662)(2712,5262)(3312,5262)(2862,5562)
\path(912,5262)(1812,4662)(1062,5562)(1512,5262)
\path(2562,6462)(1662,7062)(2262,7062)
	(1812,7362)(2562,6462)
\path(5712,6462)(4812,7062)(5412,7062)(4512,6462)
\path(4962,7362)(5412,7062)
\path(3912,7962)(4812,8562)(4212,8562)(5112,7962)
\path(4812,8562)(4362,8862)(4512,8562)(4512,8412)
\path(4212,9462)(3312,10062)(3912,10062)(3012,9462)
\path(3912,10062)(3462,10362)(3612,10062)(3612,9912)
\path(4212,9462)(3612,10062)
\path(3312,7962)(2412,8562)(3012,8562)
	(2562,8862)(2712,8562)(3312,7962)
\path(2712,8562)(2712,8412)(2112,7962)
\path(3162,1212)(4062,762)
\path(3941.252,788.833)(4062.000,762.000)(3968.085,842.498)
\path(6762,1212)(5712,612)
\path(5801.305,697.584)(5712.000,612.000)(5831.073,645.489)
\path(762,2862)(1662,2262)
\path(1545.513,2303.603)(1662.000,2262.000)(1578.795,2353.526)
\path(1212,4512)(1062,4062)
\path(1071.487,4185.329)(1062.000,4062.000)(1128.408,4166.355)
\path(1962,6312)(1662,5862)
\path(1703.603,5978.487)(1662.000,5862.000)(1753.526,5945.205)
\path(2712,7812)(2412,7362)
\path(2453.603,7478.487)(2412.000,7362.000)(2503.526,7445.205)
\path(3462,9312)(3162,8862)
\path(3203.603,8978.487)(3162.000,8862.000)(3253.526,8945.205)
\path(3762,9312)(4062,8862)
\path(3970.474,8945.205)(4062.000,8862.000)(4020.397,8978.487)
\path(4512,7812)(4812,7362)
\path(4720.474,7445.205)(4812.000,7362.000)(4770.397,7478.487)
\path(5112,6312)(5412,5862)
\path(5320.474,5945.205)(5412.000,5862.000)(5370.397,5978.487)
\path(5712,4362)(6012,3912)
\path(5920.474,3995.205)(6012.000,3912.000)(5970.397,4028.487)
\path(6462,2862)(6762,2412)
\path(6670.474,2495.205)(6762.000,2412.000)(6720.397,2528.487)
\path(5262,4362)(4662,3912)
\path(4740.000,4008.000)(4662.000,3912.000)(4776.000,3960.000)
\path(3612,2862)(2862,2262)
\path(2936.963,2360.389)(2862.000,2262.000)(2974.445,2313.537)
\path(3312,4512)(3612,4062)
\path(3520.474,4145.205)(3612.000,4062.000)(3570.397,4178.487)
\path(2262,6312)(2562,5862)
\path(2470.474,5945.205)(2562.000,5862.000)(2520.397,5978.487)
\path(4512,6312)(3462,5562)
\path(3542.211,5656.161)(3462.000,5562.000)(3577.085,5607.337)
\end{picture}
}\\
(a)
\end{minipage}
\hspace{0.5cm}
 \begin{minipage}[b]{2.5in} 
\centering
\setlength{\unitlength}{0.00029167in}
\begingroup\makeatletter\ifx\SetFigFontNFSS\undefined%
\gdef\SetFigFontNFSS#1#2#3#4#5{%
  \reset@font\fontsize{#1}{#2pt}%
  \fontfamily{#3}\fontseries{#4}\fontshape{#5}%
  \selectfont}%
\fi\endgroup%
{\renewcommand{\dashlinestretch}{30}
\begin{picture}(7824,7689)(0,-10)
\path(4062,12)(4662,1212)(5262,12)(4062,12)
\path(2412,4662)(3012,5862)(3612,4662)(2412,4662)
\path(1362,6462)(1962,7662)(2562,6462)(1362,6462)
\path(5112,4662)(5712,5862)(6312,4662)(5112,4662)
\path(5862,3162)(6462,4362)(7062,3162)(5862,3162)
\path(6612,1512)(7212,2712)(7812,1512)(6612,1512)
\path(12,3162)(612,4362)(1212,3162)(12,3162)
\path(612,4662)(1212,5862)(1812,4662)(612,4662)
\path(1662,1512)(2262,2712)(2862,1512)(1662,1512)
\path(1962,2112)(2862,1512)
\path(7512,2112)(6612,1512)
\path(3162,1212)(4062,762)
\path(3941.252,788.833)(4062.000,762.000)(3968.085,842.498)
\path(6762,1212)(5712,612)
\path(5801.305,697.584)(5712.000,612.000)(5831.073,645.489)
\path(762,2862)(1662,2262)
\path(1545.513,2303.603)(1662.000,2262.000)(1578.795,2353.526)
\path(1212,4512)(1062,4062)
\path(1071.487,4185.329)(1062.000,4062.000)(1128.408,4166.355)
\path(1962,6312)(1662,5862)
\path(1703.603,5978.487)(1662.000,5862.000)(1753.526,5945.205)
\path(5712,4362)(6012,3912)
\path(5920.474,3995.205)(6012.000,3912.000)(5970.397,4028.487)
\path(6462,2862)(6762,2412)
\path(6670.474,2495.205)(6762.000,2412.000)(6720.397,2528.487)
\path(5262,4362)(4662,3912)
\path(4740.000,4008.000)(4662.000,3912.000)(4776.000,3960.000)
\path(3612,2862)(2862,2262)
\path(2936.963,2360.389)(2862.000,2262.000)(2974.445,2313.537)
\path(3312,4512)(3612,4062)
\path(3520.474,4145.205)(3612.000,4062.000)(3570.397,4178.487)
\path(2262,6312)(2562,5862)
\path(2470.474,5945.205)(2562.000,5862.000)(2520.397,5978.487)
\path(6762,3762)(5862,3162)
\path(6462,4362)(6462,3537)(7062,3162)
\path(6012,5262)(5112,4662)
\path(5412,5262)(6312,4662)
\path(5712,5862)(5712,5037)
\path(1212,3162)(312,3762)(912,3762)
\path(3612,4662)(2712,5262)
\path(3012,5787)(3012,5112)(2412,4662)
\path(2712,5262)(3012,5262)(3612,4662)
\path(1512,5262)(912,5262)(1812,4662)
	(1212,5262)(1212,5862)
\path(2262,7062)(1662,7062)(2562,6462)
	(1962,7062)(1962,7662)
\path(1962,7062)(1962,6837)(1362,6462)
\texture{44555555 55aaaaaa aa555555 55aaaaaa aa555555 55aaaaaa aa555555 55aaaaaa 
	aa555555 55aaaaaa aa555555 55aaaaaa aa555555 55aaaaaa aa555555 55aaaaaa 
	aa555555 55aaaaaa aa555555 55aaaaaa aa555555 55aaaaaa aa555555 55aaaaaa 
	aa555555 55aaaaaa aa555555 55aaaaaa aa555555 55aaaaaa aa555555 55aaaaaa }
\path(4662,3162)(3762,3762)
\path(4662,3162)(3762,3762)
\shade\path(4062,4362)(4062,3537)(4662,3162)(4062,4362)
\path(4062,4362)(4062,3537)(4662,3162)(4062,4362)
\path(3462,3162)(4062,4362)(4662,3162)(3462,3162)
\path(4062,3537)(3462,3162)
\path(4062,3537)(3462,3162)
\end{picture}
}\\
(b)
\end{minipage}
\caption{Examples of factorization using (a) algorithm A and (b) algorithm B.}
\label{fig-example}
\end{figure}  

Figure~\ref{fig-example}(b) shows the first two steps of applying
algorithm $B$ to a sequence of two star assemblies followed by one
star subdivision. One can see that after replacing the original cone with the shaded one in the figure, we are back to the situation we started with. It follows that algorithm $B$ applied to the two star assemblies and one star subdivision in Figure~\ref{fig-example}(b) will run into a cycle and never finish. However, we have not found any such infinite loops in the case of algorithm A. Therefore we can state:

\begin{conjecture}\label{con-global-finite} 
Algorithm $A$ is always finite.
\end{conjecture}

When constructing examples of factorizations using algorithm $A$, the most complicated ones are similar to the one in Figure~\ref{fig-example}(a). We take two sequences of star subdivisions of a single cone generated by $v_1, v_2, v_3$. On one side we star subdivide at the rays generated by $v_1+v_2$, $2v_1+ v_2$, $\ldots$, $m v_1+v_2$, and on the other side we subdivide at the rays generated by $v_1+v_3$, $2v_1+ v_3$, $\ldots$, $n v_1+v_3$. The example in  Figure~\ref{fig-example}(a) shows the case when $m=2$ and $n=1$. When both $m=n=10$, the number of star subdivisions in the diagram will be in the thousands. When $m=n=40$, the number of star subdivisions needed will be in the hundreds of thousands.

One of the main results we prove is that algorithm $A$ is finite on the diagrams with $m$ star assemblies and $n$ star subdivisions as described above. We give a precise pattern for the cones appearing in the common refinement. On other types of diagrams the factorization algorithm may be shorter, but we can not say anything about the regularity or patterns that may occur. As a result we cannot prove finiteness in general.

To study algorithm $A$, we reduce it to the {\em local} case and prove that finiteness of the local algorithm implies finiteness of the global one. The local algorithm is more algebraic. It can be applied to sequences of symbols instead of drawing pictures, and it can also be easily implemented on a computer.

We will work only with fans in dimension $3$. The algorithm, however, also applies to fans in dimension greater than $3$. In higher dimensions we can again assume that all subdivision rays lie in $2$-dimensional cones. Then two star subdivisions commute unless their subdivision rays lie in the same $3$-dimensional cone. In the latter case we can apply the algorithm as in the $3$-dimensional case. Everything we say below for $3$-dimensional fans is also true, with minimal modifications, for higher dimensinal fans.

{\bf Acknowledgements.} The two algorithms discussed in this article are certainly not new and have been studied by many people. The second author would especially like to thank Dan Abramovich, Kenji Matsuki and Jaroslaw W{\l}odarczyk for fruitful discussion regarding these algorithms and their possible extensions to non-toric situations.
 
\section{The local algorithm}

\subsection{Fans and star subdivisions}

We refer to \cite{F,O} for background material about fans and toric varieties. 

We only consider $3$-dimensional fans in $\RR^3$ where all maximal cones have dimension $3$. A nonsingular cone $\sigma = \langle v_1, v_2, v_3 \rangle$ is generated by a basis $v_1, v_2, v_3$ of the lattice $\ZZ^3 \subset \RR^3$. A nonsingular fan has all its maximal cones nonsingular. A star subdivision of a nonsingular fan is called {\em smooth} if the resulting fan is again nonsingular. The inverse of a smooth star subdivision is called a smooth star assembly.

There are two types of star subdivisions of $3$-dimensional fans --
the subdivision ray can be in the interior of a $3$-dimensional cone,
or in a $2$-dimensional cone. We can always replace the first type of
star subdivision by a sequence of star subdivisions and assemblies of
the second type (see Figure~\ref{fig-replace}). A smooth star subdivision of a cone $\sigma = \langle v_1, v_2, v_3 \rangle$  of the second type has its subdivision ray generated by $v_i+v_j$ for $i\neq j$.

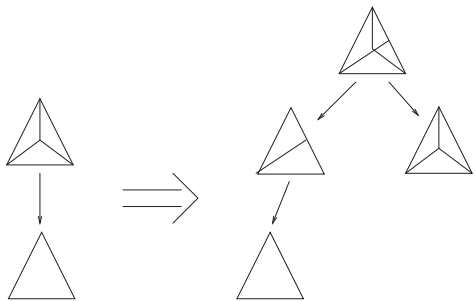
\begin{figure}[h]
\centerline{\setlength{\unitlength}{0.00029167in}
\begingroup\makeatletter\ifx\SetFigFontNFSS\undefined%
\gdef\SetFigFontNFSS#1#2#3#4#5{%
  \reset@font\fontsize{#1}{#2pt}%
  \fontfamily{#3}\fontseries{#4}\fontshape{#5}%
  \selectfont}%
\fi\endgroup%
{\renewcommand{\dashlinestretch}{30}
\begin{picture}(8424,5304)(0,-10)
\path(4167,12)(4767,1212)(5367,12)(4167,12)
\path(12,2427)(612,3627)(1212,2427)(12,2427)
\path(42,12)(642,1212)(1242,12)(42,12)
\path(12,2427)(612,2877)(1212,2427)
\path(612,3627)(612,2877)
\path(4542,2262)(5142,3462)(5742,2262)(4542,2262)
\path(6012,4077)(6612,5277)(7212,4077)(6012,4077)
\path(5412,2877)(4512,2277)
\path(6912,4677)(6012,4077)
\path(6612,5277)(6612,4527)(7212,4077)
\path(7212,2277)(7812,3477)(8412,2277)(7212,2277)
\path(7212,2277)(7812,2727)(8412,2277)
\path(7812,3477)(7812,2727)
\path(612,2277)(612,1377)
\path(582.000,1497.000)(612.000,1377.000)(642.000,1497.000)
\path(5112,2127)(4812,1377)
\path(4828.713,1499.559)(4812.000,1377.000)(4884.421,1477.275)
\path(2112,1977)(3162,1977)
\path(2112,1677)(3162,1677)
\path(3012,2277)(3462,1827)(3012,1377)
\path(6312,3927)(5637,3252)
\path(5700.640,3358.066)(5637.000,3252.000)(5743.066,3315.640)
\path(6912,3927)(7437,3327)
\path(7335.402,3397.554)(7437.000,3327.000)(7380.557,3437.064)
\end{picture}
}}
\caption{Replacing one star subdivision by a sequence of subdivisions and assemblies.} \label{fig-replace}
\end{figure}  

\subsection{The global algorithm}
 
Recall from the introduction that we start with a sequence of nonsingular fans, connected by smooth star subdivisions. We may assume that all subdivision rays lie in $2$-dimensional cones. This property is preserved after applying one step of algorithm $A$, hence we will only consider star subdivisions and assemblies of this type.
The algorithm terminates if all star subdivisions precede star assemblies; in other words, when we have a strong factorization. 

At each step of the algorithm there may be many choices of pairs of a star assembly followed by a star subdivision that we wish to commute. To make the algorithm not depend on any choices, we need to fix one ordering of such pairs, for example we can insist that always the leftmost pair is commuted. However,  finiteness of the algorithm or its end result (in case it is finite) does not depend on the chosen order. 

It is also clear that the algorithm is finite if it is finite when applied to a sequence of star assemblies followed by a sequence of star subdivisions (or just one star subdivision). Thus we may consider a single fan $\Delta$ and two sequences of smooth star subdivisions of this fan. If the algorithm is finite, it will produce extensions of these two sequences resulting in a common refinement.

\subsection {Localization}

To localize the algorithm, we replace a fan by a single cone and a
star subdivision by a subdivision of the cone together with a choice of a cone in the subdivided fan. When drawing pictures of local subdivisions, we indicate the chosen cone by shading it. 

Now given two local subdivisions of the same cone, we can use the global algorithm to construct a local factorization. Figure~\ref{fig-local-example} shows an example of such a local factorization. Notice that the factorization of the two initial local subdivisions in this example is unique: there is a unique choice of cone for each subdivision provided by algorithm $A$. 
 
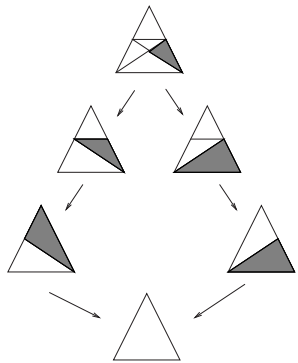
\begin{figure}[h]
\centerline{\setlength{\unitlength}{0.00029167in}
\begingroup\makeatletter\ifx\SetFigFontNFSS\undefined%
\gdef\SetFigFontNFSS#1#2#3#4#5{%
  \reset@font\fontsize{#1}{#2pt}%
  \fontfamily{#3}\fontseries{#4}\fontshape{#5}%
  \selectfont}%
\fi\endgroup%
{\renewcommand{\dashlinestretch}{30}
\begin{picture}(5199,6429)(0,-10)
\path(1917,12)(2517,1212)(3117,12)(1917,12)
\path(3987,1602)(4587,2802)(5187,1602)(3987,1602)
\path(762,1227)(1662,777)
\path(1541.252,803.833)(1662.000,777.000)(1568.085,857.498)
\path(4287,1377)(3387,777)
\path(3470.205,868.526)(3387.000,777.000)(3503.487,818.603)
\path(4887,2202)(3987,1602)
\path(1962,5202)(2562,6402)(3162,5202)(1962,5202)
\path(1962,5202)(2862,5802)(2262,5802)(3162,5202)
\path(3012,3402)(3612,4602)(4212,3402)(3012,3402)
\path(3012,3402)(3912,4002)(3312,4002)
\path(912,3402)(1512,4602)(2112,3402)(912,3402)
\path(2112,3402)(1212,4002)(1812,4002)
\path(312,2202)(1212,1602)
\path(12,1602)(612,2802)(1212,1602)(12,1602)
\path(2297,4879)(1997,4429)
\path(2038.603,4545.487)(1997.000,4429.000)(2088.526,4512.205)
\path(2862,4902)(3162,4452)
\path(3070.474,4535.205)(3162.000,4452.000)(3120.397,4568.487)
\path(1362,3177)(1062,2727)
\path(1103.603,2843.487)(1062.000,2727.000)(1153.526,2810.205)
\path(3837,3177)(4137,2727)
\path(4045.474,2810.205)(4137.000,2727.000)(4095.397,2843.487)
\texture{44555555 55aaaaaa aa555555 55aaaaaa aa555555 55aaaaaa aa555555 55aaaaaa 
	aa555555 55aaaaaa aa555555 55aaaaaa aa555555 55aaaaaa aa555555 55aaaaaa 
	aa555555 55aaaaaa aa555555 55aaaaaa aa555555 55aaaaaa aa555555 55aaaaaa 
	aa555555 55aaaaaa aa555555 55aaaaaa aa555555 55aaaaaa aa555555 55aaaaaa }
\shade\path(612,2802)(312,2202)(1212,1602)
	(612,2802)(612,2802)
\path(612,2802)(312,2202)(1212,1602)
	(612,2802)(612,2802)
\shade\path(4887,2202)(3987,1602)(5187,1602)
	(4887,2202)(4887,2202)
\path(4887,2202)(3987,1602)(5187,1602)
	(4887,2202)(4887,2202)
\shade\path(1812,4002)(1212,4002)(2112,3402)
	(1812,4002)(1812,4002)
\path(1812,4002)(1212,4002)(2112,3402)
	(1812,4002)(1812,4002)
\shade\path(3912,4002)(3012,3402)(4212,3402)
	(3912,4002)(3912,4002)
\path(3912,4002)(3012,3402)(4212,3402)
	(3912,4002)(3912,4002)
\shade\path(2862,5802)(2562,5577)(3162,5202)
	(2862,5802)(2862,5802)
\path(2862,5802)(2562,5577)(3162,5202)
	(2862,5802)(2862,5802)
\end{picture}
}}
\caption{Local factorization using algorithm $A$.}
\label{fig-local-example}
\end{figure}  

Figure~\ref{fig-local-example2} shows two different local factorizations of one pair of initial local subdivisions. In this case algorithm $A$ provides two choices of local factorizations. When factoring a sequence that is longer than two star subdivisions, at each step of applying algorithm $A$ we have one or two choices of local factorization. As a result, we get in general many local factorizations of one initial sequence.

\begin{figure}[h]
\centerline{\setlength{\unitlength}{0.00029167in}
\begingroup\makeatletter\ifx\SetFigFontNFSS\undefined%
\gdef\SetFigFontNFSS#1#2#3#4#5{%
  \reset@font\fontsize{#1}{#2pt}%
  \fontfamily{#3}\fontseries{#4}\fontshape{#5}%
  \selectfont}%
\fi\endgroup%
{\renewcommand{\dashlinestretch}{30}
\begin{picture}(11424,6429)(0,-10)
\path(8142,12)(8742,1212)(9342,12)(8142,12)
\path(10212,1602)(10812,2802)(11412,1602)(10212,1602)
\path(6987,1227)(7887,777)
\path(7766.252,803.833)(7887.000,777.000)(7793.085,857.498)
\path(10512,1377)(9612,777)
\path(9695.205,868.526)(9612.000,777.000)(9728.487,818.603)
\path(11112,2202)(10212,1602)
\path(8187,5202)(8787,6402)(9387,5202)(8187,5202)
\path(8187,5202)(9087,5802)(8487,5802)(9387,5202)
\path(9237,3402)(9837,4602)(10437,3402)(9237,3402)
\path(9237,3402)(10137,4002)(9537,4002)
\path(7137,3402)(7737,4602)(8337,3402)(7137,3402)
\path(8337,3402)(7437,4002)(8037,4002)
\path(6537,2202)(7437,1602)
\path(6237,1602)(6837,2802)(7437,1602)(6237,1602)
\path(8522,4879)(8222,4429)
\path(8263.603,4545.487)(8222.000,4429.000)(8313.526,4512.205)
\path(9087,4902)(9387,4452)
\path(9295.474,4535.205)(9387.000,4452.000)(9345.397,4568.487)
\path(7587,3177)(7287,2727)
\path(7328.603,2843.487)(7287.000,2727.000)(7378.526,2810.205)
\path(10062,3177)(10362,2727)
\path(10270.474,2810.205)(10362.000,2727.000)(10320.397,2843.487)
\texture{44555555 55aaaaaa aa555555 55aaaaaa aa555555 55aaaaaa aa555555 55aaaaaa 
	aa555555 55aaaaaa aa555555 55aaaaaa aa555555 55aaaaaa aa555555 55aaaaaa 
	aa555555 55aaaaaa aa555555 55aaaaaa aa555555 55aaaaaa aa555555 55aaaaaa 
	aa555555 55aaaaaa aa555555 55aaaaaa aa555555 55aaaaaa aa555555 55aaaaaa }
\shade\path(6837,2802)(6537,2202)(7437,1602)
	(6837,2802)(6837,2802)
\path(6837,2802)(6537,2202)(7437,1602)
	(6837,2802)(6837,2802)
\shade\path(10212,1602)(11112,2202)(10812,2802)(10212,1602)
\path(10212,1602)(11112,2202)(10812,2802)(10212,1602)
\path(1917,12)(2517,1212)(3117,12)(1917,12)
\path(3987,1602)(4587,2802)(5187,1602)(3987,1602)
\path(762,1227)(1662,777)
\path(1541.252,803.833)(1662.000,777.000)(1568.085,857.498)
\path(4287,1377)(3387,777)
\path(3470.205,868.526)(3387.000,777.000)(3503.487,818.603)
\path(4887,2202)(3987,1602)
\path(1962,5202)(2562,6402)(3162,5202)(1962,5202)
\path(1962,5202)(2862,5802)(2262,5802)(3162,5202)
\path(3012,3402)(3612,4602)(4212,3402)(3012,3402)
\path(3012,3402)(3912,4002)(3312,4002)
\path(912,3402)(1512,4602)(2112,3402)(912,3402)
\path(2112,3402)(1212,4002)(1812,4002)
\path(312,2202)(1212,1602)
\path(12,1602)(612,2802)(1212,1602)(12,1602)
\path(2297,4879)(1997,4429)
\path(2038.603,4545.487)(1997.000,4429.000)(2088.526,4512.205)
\path(2862,4902)(3162,4452)
\path(3070.474,4535.205)(3162.000,4452.000)(3120.397,4568.487)
\path(1362,3177)(1062,2727)
\path(1103.603,2843.487)(1062.000,2727.000)(1153.526,2810.205)
\path(3837,3177)(4137,2727)
\path(4045.474,2810.205)(4137.000,2727.000)(4095.397,2843.487)
\shade\path(612,2802)(312,2202)(1212,1602)
	(612,2802)(612,2802)
\path(612,2802)(312,2202)(1212,1602)
	(612,2802)(612,2802)
\shade\path(1812,4002)(1212,4002)(2112,3402)
	(1812,4002)(1812,4002)
\path(1812,4002)(1212,4002)(2112,3402)
	(1812,4002)(1812,4002)
\shade\path(3987,1602)(4887,2202)(4587,2802)(3987,1602)
\path(3987,1602)(4887,2202)(4587,2802)(3987,1602)
\shade\path(3912,4002)(3312,4002)(3012,3402)
	(3912,4002)(3912,4002)
\path(3912,4002)(3312,4002)(3012,3402)
	(3912,4002)(3912,4002)
\shade\path(2262,5802)(2562,5577)(2862,5802)
	(2262,5802)(2262,5802)
\path(2262,5802)(2562,5577)(2862,5802)
	(2262,5802)(2262,5802)
\shade\path(7737,4602)(7437,4002)(8037,4002)
	(7737,4602)(7737,4602)
\path(7737,4602)(7437,4002)(8037,4002)
	(7737,4602)(7737,4602)
\shade\path(9837,4602)(9537,4002)(10137,4002)
	(9837,4602)(9837,4602)
\path(9837,4602)(9537,4002)(10137,4002)
	(9837,4602)(9837,4602)
\shade\path(8787,6402)(8487,5802)(9087,5802)
	(8787,6402)(8787,6402)
\path(8787,6402)(8487,5802)(9087,5802)
	(8787,6402)(8787,6402)
\end{picture}
}}
\caption{Two local factorizations of the same initial sequence.}
\label{fig-local-example2}
\end{figure}
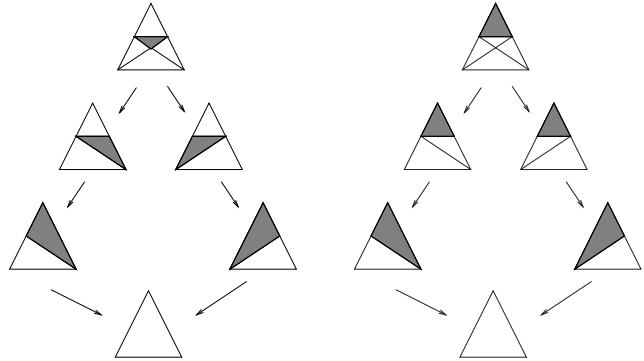  

A local star subdivision can be represented by a matrix as
follows. Consider a cone $\langle v_1, v_2, v_3\rangle$ and its local
star subdivision resulting in the new cone $\langle v_1, v_1+v_2, v_3\rangle$. Let $M$ be the $3\times 3$ matrix with columns $ v_1, v_2, v_3$ and $N$ the matrix with columns $\langle v_1, v_1+v_2, v_3\rangle$. Then 
\[ N = M E_{1 2},\]
where $E_{1 2}$ is an elementary matrix that differs from the identity matrix by the entry $1$ at position $(1,2)$. In the same way all $6$ possible local subdivisions of the cone $\langle v_1, v_2, v_3\rangle$ can be represented by elementary matrices $E_{ij}$ for $i,j\in\{1,2,3\}, i\neq j$. Star assemblies are represented by inverses $E_{ij}^{-1}$ of these elementary matrices.

To understand the local factorization algorithm in terms of matrices, consider  Figure~\ref{fig-local-example-mat} where we have labeled the local star subdivisions by elementary matrices. The factorization replaces the bottom of the diagram, which we read from left to right as $E_{12}^{-1} E_{31}$, with the top of the diagram $E_{31} E_{32} E_{12}^{-1}$. We denote this replacement as:
\[ E_{12}^{-1} E_{31} \go E_{31} E_{32} E_{12}^{-1}.\]
(One can recognize this as an actual equality between products of elementary matrices, but we use the symbol $\go$ to indicate the direction in which the replacement is done.)

\begin{figure}[h]
\centerline{\setlength{\unitlength}{0.00029167in}
\begingroup\makeatletter\ifx\SetFigFontNFSS\undefined%
\gdef\SetFigFontNFSS#1#2#3#4#5{%
  \reset@font\fontsize{#1}{#2pt}%
  \fontfamily{#3}\fontseries{#4}\fontshape{#5}%
  \selectfont}%
\fi\endgroup%
{\renewcommand{\dashlinestretch}{30}
\begin{picture}(5232,6582)(0,-10)
\path(1950,165)(2550,1365)(3150,165)(1950,165)
\path(4020,1755)(4620,2955)(5220,1755)(4020,1755)
\path(795,1380)(1695,930)
\path(1574.252,956.833)(1695.000,930.000)(1601.085,1010.498)
\path(4320,1530)(3420,930)
\path(3503.205,1021.526)(3420.000,930.000)(3536.487,971.603)
\path(4920,2355)(4020,1755)
\path(1995,5355)(2595,6555)(3195,5355)(1995,5355)
\path(1995,5355)(2895,5955)(2295,5955)(3195,5355)
\path(3045,3555)(3645,4755)(4245,3555)(3045,3555)
\path(3045,3555)(3945,4155)(3345,4155)
\path(945,3555)(1545,4755)(2145,3555)(945,3555)
\path(2145,3555)(1245,4155)(1845,4155)
\path(345,2355)(1245,1755)
\path(45,1755)(645,2955)(1245,1755)(45,1755)
\texture{44555555 55aaaaaa aa555555 55aaaaaa aa555555 55aaaaaa aa555555 55aaaaaa 
	aa555555 55aaaaaa aa555555 55aaaaaa aa555555 55aaaaaa aa555555 55aaaaaa 
	aa555555 55aaaaaa aa555555 55aaaaaa aa555555 55aaaaaa aa555555 55aaaaaa 
	aa555555 55aaaaaa aa555555 55aaaaaa aa555555 55aaaaaa aa555555 55aaaaaa }
\shade\path(645,2955)(345,2355)(1245,1755)(645,2955)
\path(645,2955)(345,2355)(1245,1755)(645,2955)
\shade\path(4920,2355)(4020,1755)(5220,1755)(4920,2355)
\path(4920,2355)(4020,1755)(5220,1755)(4920,2355)
\shade\path(1845,4155)(1245,4155)(2145,3555)(1845,4155)
\path(1845,4155)(1245,4155)(2145,3555)(1845,4155)
\shade\path(3945,4155)(3045,3555)(4245,3555)(3945,4155)
\path(3945,4155)(3045,3555)(4245,3555)(3945,4155)
\shade\path(2895,5955)(2595,5730)(3195,5355)(2895,5955)
\path(2895,5955)(2595,5730)(3195,5355)(2895,5955)
\path(1280,3382)(980,2932)
\path(1021.603,3048.487)(980.000,2932.000)(1071.526,3015.205)
\path(3985,3382)(4285,2932)
\path(4193.474,3015.205)(4285.000,2932.000)(4243.397,3048.487)
\path(3045,5055)(3345,4605)
\path(3253.474,4688.205)(3345.000,4605.000)(3303.397,4721.487)
\path(2145,5055)(1845,4605)
\path(1886.603,4721.487)(1845.000,4605.000)(1936.526,4688.205)
\put(2415,1590){\makebox(0,0)[lb]{\smash{{\SetFigFontNFSS{5}{6.0}{\rmdefault}{\mddefault}{\updefault}V1}}}}
\put(3300,60){\makebox(0,0)[lb]{\smash{{\SetFigFontNFSS{5}{6.0}{\rmdefault}{\mddefault}{\updefault}V3}}}}
\put(4050,870){\makebox(0,0)[lb]{\smash{{\SetFigFontNFSS{8}{9.6}{\rmdefault}{\mddefault}{\updefault}E}}}}
\put(570,855){\makebox(0,0)[lb]{\smash{{\SetFigFontNFSS{8}{9.6}{\rmdefault}{\mddefault}{\updefault}E}}}}
\put(15,3150){\makebox(0,0)[lb]{\smash{{\SetFigFontNFSS{8}{9.6}{\rmdefault}{\mddefault}{\updefault}E}}}}
\put(3540,5040){\makebox(0,0)[lb]{\smash{{\SetFigFontNFSS{8}{9.6}{\rmdefault}{\mddefault}{\updefault}E}}}}
\put(1215,5070){\makebox(0,0)[lb]{\smash{{\SetFigFontNFSS{8}{9.6}{\rmdefault}{\mddefault}{\updefault}E}}}}
\put(840,735){\makebox(0,0)[lb]{\smash{{\SetFigFontNFSS{5}{6.0}{\rmdefault}{\mddefault}{\updefault}12}}}}
\put(4350,735){\makebox(0,0)[lb]{\smash{{\SetFigFontNFSS{5}{6.0}{\rmdefault}{\mddefault}{\updefault}31}}}}
\put(315,3105){\makebox(0,0)[lb]{\smash{{\SetFigFontNFSS{5}{6.0}{\rmdefault}{\mddefault}{\updefault}31}}}}
\put(4785,3120){\makebox(0,0)[lb]{\smash{{\SetFigFontNFSS{8}{9.6}{\rmdefault}{\mddefault}{\updefault}Id}}}}
\put(3825,4965){\makebox(0,0)[lb]{\smash{{\SetFigFontNFSS{5}{6.0}{\rmdefault}{\mddefault}{\updefault}12}}}}
\put(1485,4995){\makebox(0,0)[lb]{\smash{{\SetFigFontNFSS{5}{6.0}{\rmdefault}{\mddefault}{\updefault}32}}}}
\put(1470,15){\makebox(0,0)[lb]{\smash{{\SetFigFontNFSS{5}{6.0}{\rmdefault}{\mddefault}{\updefault}V2}}}}
\end{picture}
}}
\caption{Local factorization represented by matrices.}
\label{fig-local-example-mat}
\end{figure}
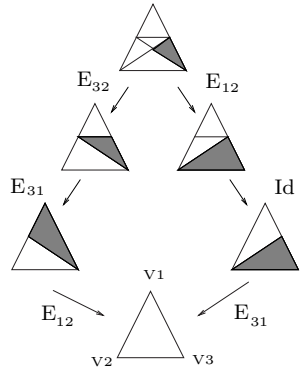  

The local factorization algorithm can now be described as follows. We start with a sequence of elementary matrices and their inverses. At each step we look for a pair $E_{ij}^{-1} E_{kl}$ in the sequence, and the algorithm tells us how to commute these matrices, possibly inserting new elementary matrices. The algorithm is finished when all elementary matrices lie to the left of the inverses.

One problem with the local factorization algorithm that we haven't
discussed is that the matrix representation $M$ of a cone $\langle
v_1, v_2, v_3\rangle$ depends on the ordering of the generators. However, if we choose one ordering of generators, then we automatically get an ordering of generators after one local star subdivision and hence the elementary matrix that represents this subdivision. Now the question is, if we have two sequences of local star subdivisions starting and ending with the same cone, do the orderings induced from both sequences agree at the final cone? The answer is ``no'' in general. Figure~\ref{fig-local-example-order} shows a local factorization step prescribed by algorithm $A$, where there is no consistent ordering of generators in all cones. The two orderings of generators in the top cone are cyclic permutations of each other. We write this factorization as:
\[ E_{12}^{-1} E_{13} \go E_{31} E_{23} R_{321} E_{32}^{-1} E_{21}^{-1} .\]
Here $R_{321}$ is a permutation matrix that represents the cyclic permutation of generators.

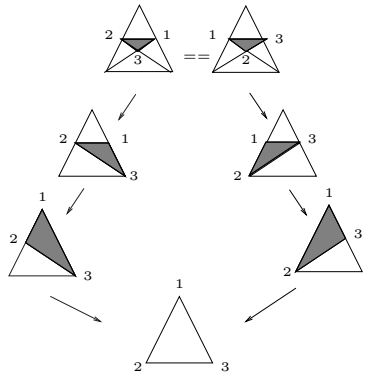
\begin{figure}[h]
\centerline{\setlength{\unitlength}{0.00029167in}
\begingroup\makeatletter\ifx\SetFigFontNFSS\undefined%
\gdef\SetFigFontNFSS#1#2#3#4#5{%
  \reset@font\fontsize{#1}{#2pt}%
  \fontfamily{#3}\fontseries{#4}\fontshape{#5}%
  \selectfont}%
\fi\endgroup%
{\renewcommand{\dashlinestretch}{30}
\begin{picture}(6402,6657)(0,-10)
\path(765,1365)(1665,915)
\path(1544.252,941.833)(1665.000,915.000)(1571.085,995.498)
\path(2115,3540)(1215,4140)(1815,4140)
\path(315,2340)(1215,1740)
\path(15,1740)(615,2940)(1215,1740)(15,1740)
\path(2300,5017)(2000,4567)
\path(2041.603,4683.487)(2000.000,4567.000)(2091.526,4650.205)
\path(1365,3315)(1065,2865)
\path(1106.603,2981.487)(1065.000,2865.000)(1156.526,2948.205)
\texture{44555555 55aaaaaa aa555555 55aaaaaa aa555555 55aaaaaa aa555555 55aaaaaa 
	aa555555 55aaaaaa aa555555 55aaaaaa aa555555 55aaaaaa aa555555 55aaaaaa 
	aa555555 55aaaaaa aa555555 55aaaaaa aa555555 55aaaaaa aa555555 55aaaaaa 
	aa555555 55aaaaaa aa555555 55aaaaaa aa555555 55aaaaaa aa555555 55aaaaaa }
\shade\path(615,2940)(315,2340)(1215,1740)(615,2940)
\path(615,2940)(315,2340)(1215,1740)(615,2940)
\shade\path(1815,4140)(1215,4140)(2115,3540)(1815,4140)
\path(1815,4140)(1215,4140)(2115,3540)(1815,4140)
\path(2490,165)(3090,1365)(3690,165)(2490,165)
\path(5190,1815)(5790,3015)(6390,1815)(5190,1815)
\shade\path(5190,1815)(6090,2415)(5790,3015)(5190,1815)
\path(5190,1815)(6090,2415)(5790,3015)(5190,1815)
\path(4365,3540)(4965,4740)(5565,3540)(4365,3540)
\path(4365,3540)(5265,4140)(4665,4140)
\shade\path(5253,4157)(4653,4157)(4353,3557)(5253,4157)
\path(5253,4157)(4653,4157)(4353,3557)(5253,4157)
\path(1770,5430)(2370,6630)(2970,5430)(1770,5430)
\shade\path(2040,6015)(2340,5790)(2640,6015)(2040,6015)
\path(2040,6015)(2340,5790)(2640,6015)(2040,6015)
\path(1740,5415)(2640,6015)(2040,6015)(2940,5415)
\path(5080,3292)(5380,2842)
\path(5288.474,2925.205)(5380.000,2842.000)(5338.397,2958.487)
\path(4365,5040)(4665,4590)
\path(4573.474,4673.205)(4665.000,4590.000)(4623.397,4706.487)
\path(5190,1515)(4290,915)
\path(4373.205,1006.526)(4290.000,915.000)(4406.487,956.603)
\path(3690,5415)(4290,6615)(4890,5415)(3690,5415)
\path(3687,5419)(4662,6019)(3987,6019)(4887,5419)
\shade\path(3990,6015)(4290,5790)(4665,6015)
	(3990,6015)(3990,6015)
\path(3990,6015)(4290,5790)(4665,6015)
	(3990,6015)(3990,6015)
\path(915,3540)(1515,4740)(2115,3540)(915,3540)
\put(2265,15){\makebox(0,0)[lb]{\smash{{\SetFigFontNFSS{5}{6.0}{\rmdefault}{\mddefault}{\updefault}2}}}}
\put(3840,15){\makebox(0,0)[lb]{\smash{{\SetFigFontNFSS{5}{6.0}{\rmdefault}{\mddefault}{\updefault}3}}}}
\put(3015,1515){\makebox(0,0)[lb]{\smash{{\SetFigFontNFSS{5}{6.0}{\rmdefault}{\mddefault}{\updefault}1}}}}
\put(540,3090){\makebox(0,0)[lb]{\smash{{\SetFigFontNFSS{5}{6.0}{\rmdefault}{\mddefault}{\updefault}1}}}}
\put(15,2340){\makebox(0,0)[lb]{\smash{{\SetFigFontNFSS{5}{6.0}{\rmdefault}{\mddefault}{\updefault}2}}}}
\put(915,4140){\makebox(0,0)[lb]{\smash{{\SetFigFontNFSS{5}{6.0}{\rmdefault}{\mddefault}{\updefault}2}}}}
\put(1740,6015){\makebox(0,0)[lb]{\smash{{\SetFigFontNFSS{5}{6.0}{\rmdefault}{\mddefault}{\updefault}2}}}}
\put(1365,1665){\makebox(0,0)[lb]{\smash{{\SetFigFontNFSS{5}{6.0}{\rmdefault}{\mddefault}{\updefault}3}}}}
\put(2190,3465){\makebox(0,0)[lb]{\smash{{\SetFigFontNFSS{5}{6.0}{\rmdefault}{\mddefault}{\updefault}3}}}}
\put(2265,5565){\makebox(0,0)[lb]{\smash{{\SetFigFontNFSS{5}{6.0}{\rmdefault}{\mddefault}{\updefault}3}}}}
\put(5715,3165){\makebox(0,0)[lb]{\smash{{\SetFigFontNFSS{5}{6.0}{\rmdefault}{\mddefault}{\updefault}1}}}}
\put(3615,6015){\makebox(0,0)[lb]{\smash{{\SetFigFontNFSS{5}{6.0}{\rmdefault}{\mddefault}{\updefault}1}}}}
\put(4065,3465){\makebox(0,0)[lb]{\smash{{\SetFigFontNFSS{5}{6.0}{\rmdefault}{\mddefault}{\updefault}2}}}}
\put(4215,5565){\makebox(0,0)[lb]{\smash{{\SetFigFontNFSS{5}{6.0}{\rmdefault}{\mddefault}{\updefault}2}}}}
\put(5415,4140){\makebox(0,0)[lb]{\smash{{\SetFigFontNFSS{5}{6.0}{\rmdefault}{\mddefault}{\updefault}3}}}}
\put(4815,5940){\makebox(0,0)[lb]{\smash{{\SetFigFontNFSS{5}{6.0}{\rmdefault}{\mddefault}{\updefault}3}}}}
\put(6240,2415){\makebox(0,0)[lb]{\smash{{\SetFigFontNFSS{5}{6.0}{\rmdefault}{\mddefault}{\updefault}3}}}}
\put(2790,6015){\makebox(0,0)[lb]{\smash{{\SetFigFontNFSS{5}{6.0}{\rmdefault}{\mddefault}{\updefault}1}}}}
\put(2040,4140){\makebox(0,0)[lb]{\smash{{\SetFigFontNFSS{5}{6.0}{\rmdefault}{\mddefault}{\updefault}1}}}}
\put(4365,4140){\makebox(0,0)[lb]{\smash{{\SetFigFontNFSS{5}{6.0}{\rmdefault}{\mddefault}{\updefault}1}}}}
\put(4965,1740){\makebox(0,0)[lb]{\smash{{\SetFigFontNFSS{5}{6.0}{\rmdefault}{\mddefault}{\updefault}2}}}}
\put(3165,5640){\makebox(0,0)[lb]{\smash{{\SetFigFontNFSS{5}{6.0}{\rmdefault}{\mddefault}{\updefault}==}}}}
\end{picture}
}}
\caption{Ordering of generators in a factorization diagram.}
\label{fig-local-example-order}
\end{figure}  

Figure~\ref{table-localA} lists the commutation rules for the local
algorithm $A$ in terms of elementary matrices. These rules apply for
$\{i,j,k\} = \{1,2,3\}$. Rule $(1)$ tells us that there is no local
factorization and we have to stop the algorithm. The matrices $E_{ij}$
and $E_{ji}$ appearing in rule $(1)$ describe the same global
subdivision but with different choices of cones. Rule $(2)$ tells us
to cancel $E_{ij}$  with its inverse. Rules $(3)-(6)$ can be read off from the global diagram of algorithm $A$ in Figure~\ref{fig-alg}. Each of these rules corresponds to one cone in the final refinement. Rule $(6)$ gives us a choice between two different factorizations; the two factorizations are shown in Figure~\ref{fig-local-example2}. Rule $(6b)$ is the only one where there is no consistent labeling of the generators and we need to use the permutation matrix $R_{kji}$. Rule $(7)$ shows how to commute the permutation matrix with the elementary matrices. 

\begin{figure}[h]
\begin{enumerate} 
\item $E^{-1}_{ij}E_{ji} \go {\text\tt stop}$.
\item $E^{-1}_{ij}E_{ij} \go 1$.              
\item $E^{-1}_{ij}E_{kj} \go E_{kj}E^{-1}_{ij}$.   
\item $E^{-1}_{ij}E_{jk} \go E_{jk}E^{-1}_{ik}E^{-1}_{ij}$.
\item $E^{-1}_{ij}E_{ki} \go E_{ki}E_{kj}E^{-1}_{ij}$.
\item[(6a)] $E^{-1}_{ij}E_{ik} \go E_{ik}E^{-1}_{ij}$  
\item[(6b)]                $E^{-1}_{ij}E_{ik} \go E_{ki}E_{jk}R_{kji}E^{-1}_{kj}E^{-1}_{ji}$.
\item[(7)] $R_{kji}E_{lm}^{\pm 1} \go E_{r(l)r(m)}^{\pm 1} R_{kji}$, $r:i\mapsto j\mapsto k\mapsto i$.
\end{enumerate}                                                                 
\caption{Rules for local algorithm $A$.}
\label{table-localA}
\end{figure}  

As explained above, we start with a sequence of elementary matrices and their inverses. The goal is to apply the commutation rules to get all elementary matrices to the left of the inverses of such matrices. We do not care about the location of the permutation matrices $R_{ijk}$; they can be moved to the right or to the left as desired. An example of applying the algorithm is:
\begin{eqnarray*} 
E_{12}^{-1}   \underline{E_{12}^{-1} E_{13}} 
&\stackrel{(6b)}{\longgo}& \underline{E_{12}^{-1} E_{31}} E_{23} R_{321} E^{-1}_{32} E^{-1}_{21} \\
&\stackrel{(5)}{\longgo}& E_{31} E_{32}\underline{E_{12}^{-1} E_{23}} R_{321} E^{-1}_{32} E^{-1}_{21} \\
 &\stackrel{(4)}{\longgo}& E_{31} E_{32}E_{23}E_{13}^{-1} E_{12}^{-1} R_{321} E^{-1}_{32} E^{-1}_{21}.
\end{eqnarray*}
At each step we have underlined the pair to which the rule is applied and the rule number is shown on the arrow. Note that at the first step we chose to apply rule $(6b)$. If at the first step we apply rule $(6a)$, then at the next step we would again have a choice between rules $(6a)$ and $(6b)$. This gives a total of three different factorizations of the initial sequence.

As in the global algorithm, there is a choice of the order in which we apply these rules. Since we want to compare the local and the global algorithms, we have to use the same order in both algorithms. For example, we can always apply the rule at the leftmost place. In the local case when applying rule $(6)$ there is also a choice between $(6a)$ and $(6b)$. Below, when talking about different choices in applying the local algorithm, we always mean the choice between $(6a)$ and $(6b)$; we assume that the order of applying the rules has been fixed.

\begin{conjecture}\label{conj-local-fin}
 The local algorithm $A$ is finite: starting with any sequence of elementary matrices and their inverses, the rules can be applied only a finite number of times for any choices between rules $(6a)$ and $(6b)$ that may occur.
\end{conjecture}

Note that the factorization algorithm ends whenever we apply rule $(1)$  {\text\tt stop}, or if there are no more places to apply the rules and we have a local strong factorization. 

The main result we prove in this section is:

\begin{proposition} \label{prop-local-global}
Conjecture~\ref{conj-local-fin} implies Conjecture~\ref{conj-oda}. 
\end{proposition}

\begin{proof}
 Let $S$ be a finite sequence of global star subdivisions and star assemblies. In other words, $S$ is a sequence of fans, each one obtained from the previous one by one star subdivision or one star assembly. A {\em local subsequence} $T$ of $S$ is a choice of a cone in each fan of $S$ such that each cone in the sequence is either equal to the previous cone, or is obtained from the previous cone by a local star subdivision or a local star assembly. A local subsequence gives a sequence of local star subdivisions and star assemblies.

A global sequence $S$ has many local subsequences. In fact, if we choose any cone in the first fan in the sequence $S$, we can extend this choice to a local subsequence $T$. Similarly, if we choose a cone in one fan in the sequence $S$, we can extend this choice to a local subsequence, going in both directions.

Now suppose $S_1$ and $S_2$ are two global sequences and $S_2$ is obtained from $S_1$ by applying one step of the factorization algorithm. The two sequences can be put into a diagram of the same type as in Figure~\ref{fig-example}. If $T_1$ is a local subsequence of $S_1$, we seek to extend it to a local subsequence $T_2$ of $S_2$. By this we mean a choice of cones for each fan of $S_2$ that agrees with the choice $T_1$ on the fans that are the same in both sequences $S_1$ and $S_2$. If such an extension $T_2$ exists, then as sequences of local star subdivisions and assemblies, $T_2$ is either equal to $T_1$ or $T_2$ is obtained from $T_1$ by one step of the local factorization algorithm. The extension of $T_1$ to $T_2$ may not always exist. However, suppose that instead of $T_1$ we are given a local subsequence $T_2$ of $S_2$, then there is always a unique extension of this subsequence to a local subsequence $T_1$ of $S_1$. The reason for this is as follows. When we star subdivide a cone, we have a choice of two cones to turn it into a local star subdivision; but if we star-assemble a cone, there is no choice at all, and it is automatically a local star assembly.  If $T_1$ is the unique extension of $T_2$, then as sequences of local star subdivisions and assemblies, $T_2$ is either equal to $T_1$ or is obtained from it by one step of the local algorithm.

Suppose we have a global sequence of star subdivisions and assemblies $S_1$ on which the algorithm is infinite. Applying the algorithm step-by-step, we construct new sequences $S_2$, $S_3$, $\ldots$. Let us also construct a graph with vertices $S_i$ and edges going from $S_i$ to $S_{i+1}$, indicating that $S_{i+1}$ is obtained from $S_i$ by one step of the algorithm. Next we construct a graph of local subdivisions. The vertices are all local subsequences of $T_i$ of $S_i$ for $i\geq 1$ and there is an edge from $T_i$ to $T_{i+1}$ if $T_i$ is a local subsequence of $S_i$ and $T_{i+1}$ is an extension of this to a local subsequence of $S_{i+1}$. By the discussion above, this graph of local subsequences is a set of trees (a forest) with roots the local subsequences of $S_1$. Since there are infinitely many $S_i$, at least one of the trees of local subsequences must contain an infinitely long  path. We claim that one of such infinitely long paths then corresponds to an infinite number of local factorization steps applied to a local subsequence of $S_1$, implying that the local factorization algorithm is not finite.

The problem with the claim above is that in the graph of trees some
edges correspond to the identity transformation. It is conceivable
that in all infinite paths all edges are eventually identities. Since
the trees have finite valence, this means that for some $N>0$ all
edges that have distance at least $N$ from the roots are
identities. To get a contradiction, we can find in some fan in $S_M$ for large $M$ a cone that is very ``small'' in the sense that it takes more than $N$ local subdivisions to reach this cone from any cone in a fan in $S_1$. We extend this choice of a cone to a local subsequence $T_M$ of $S_M$. By construction, this $T_M$ can be reached from a root by more than $N$ non-identity edges, which is a contradiction.   
\end{proof}

The proof of the previous proposition also applies to algorithm $B$ of the introduction and its corresponding local version. Since we know that the global algorithm $B$ is not finite, the same must be true for the local algorithm. Figure~\ref{table-localB} lists the rules for the local algorithm $B$. The rules are similar to the rules of algorithm $A$. In rule $(6)$ we again have a choice between $(6a)$ and $(6b)$. In this algorithm there is no need for the permutation matrix. A local factorization corresponding to  Figure~\ref{fig-example}(b) is:
\begin{eqnarray*} 
E_{13}^{-1}   \underline{E_{12}^{-1} E_{13}} 
&\stackrel{(6b)}{\longgo}& \underline{E_{13}^{-1} E_{32}} E_{13} E^{-1}_{32}  \\
 &\stackrel{(4)}{\longgo}& E_{32} E_{13}^{-1}   \underline{E_{12}^{-1} E_{13}} E^{-1}_{32}.
\end{eqnarray*}
The middle three matrices in the last sequences are the same as in the
original sequence, hence this algorithm can be repeated cyclically. 

\begin{figure}[h]
\begin{enumerate} 
\item $E^{-1}_{ij}E_{ji} \go {\text\tt stop}$.
\item $E^{-1}_{ij}E_{ij} \go 1$.              
\item $E^{-1}_{ij}E_{kj} \go E_{kj}E^{-1}_{ij}$.   
\item $E^{-1}_{ij}E_{jk} \go E_{jk}E^{-1}_{ij}E^{-1}_{ik}$.
\item $E^{-1}_{ij}E_{ki} \go E_{kj}E_{ki}E^{-1}_{ij}$.
\item[(6a)] $E^{-1}_{ij}E_{ik} \go E_{jk}E^{-1}_{ij}E^{-1}_{jk} $.  
\item[(6b)] $E^{-1}_{ij}E_{ik} \go E_{kj}E_{ik}E^{-1}_{kj}$.             
\end{enumerate}                                                                 
\caption{Rules for local algorithm $B$.}
\label{table-localB}
\end{figure}  

The local rules for the two algorithms are also valid in dimension $n>3$. We need to assume that $\{i,j,k\}$ forms a $3$-element subset of $\{1,2,\ldots,n\}$, and we need one additional rule
\[  E^{-1}_{ij}E_{kl} \go E_{kl}E^{-1}_{ij},\]
where $i,j,k,l$ are distinct.

\section{Finiteness results}

We prove finiteness of algorithm $A$ in some cases.  The term ``algorithm'' always refers to algorithm $A$. We denote by $S, T, \ldots$ sequences of elementary matrices with positive powers. We also assume everywhere that $\{i,j,k\}=\{1,2,3\}$.

\subsection{Directed sequences}

We call a sequence $S$ {\em directed toward $i$} (or simply {\em directed}) if it consists of elementary matrices $E_{ij}$ and $E_{ik}$. Thus, a directed sequence has the form
\[ S= E_{ij}^{m_1} E_{ik}^{n_1}E_{ij}^{m_2} \ldots  E_{ik}^{n_l}.\]
Our main goal in this section is to prove that if $S$ and $T$ are both directed, not necessarily toward the same $i$, then the local algorithm is finite on $S^{-1} T$.

Let us first add an extra rule to the algorithm that is useful when factoring directed sequences. The rule
 \begin{enumerate} 
\item[(8)] $E_{ij}E_{ik} \go E_{ik}E_{ij}$.   
\end{enumerate}     
allows us to commute $E_{ij}$ with $E_{ik}$. We claim that rule $(8)$ ``commutes'' with the factorization algorithm in the following sense. Suppose $U$ is a sequence of elementary matrices and their inverses and $V$ is obtained from $U$ by applying the rule $(8)$ once. Then the factorizations of $U$ are in one-to-one correspondence with factorizations of $V$ and and the latter ones are obtained from the former ones by one application of rule $(8)$. In particular, if all factorizations of $U$ are finite, then all factorizations of $V$ are also finite. To prove this claim, it suffices to compare factorizations of $E^{-1}_{\alpha\beta} E_{ij}E_{ik}$ with factorizations of $E^{-1}_{\alpha\beta} E_{ik}E_{ij}$ for different indices $\alpha,\beta$. We do the case $E_{\alpha\beta}=E_{jk}$ only and leave the other cases to the reader.
\begin{alignat*}{3}
 E^{-1}_{jk} E_{ij}E_{ik} &\go  E_{ij}E_{ik}E^{-1}_{jk}E_{ik} &\go E_{ij}E_{ik} E_{ik}E^{-1}_{jk} \\
 E^{-1}_{jk} E_{ik}E_{ij} &\go  E_{ik}E^{-1}_{jk}E_{ij} &\go E_{ik}E_{ij} E_{ik}E^{-1}_{jk} .
\end{alignat*}
Clearly the two factorizations differ by one application of rule $(8)$.

Using rule $(8)$ we can write a directed sequence as 
\[ S= E_{ij}^{m} E_{ik}^{n}.\]
Note also that rule $(8)$ keeps a directed sequence directed and an
undirected sequence undirected.

\begin{proposition} \label{prop-directedij}
 Let $S$ be a sequence directed toward $i$ and $T$ a sequence directed toward $j$, where $i\neq j$. Then $S^{-1} T$ has at most one factorization, which is of the form $T_1 S_1^{-1}$, where $S_1$ is directed toward $i$ and $T_1$ is directed toward $j$.
\end{proposition}

\begin{proof} Let 
\begin{gather*}
 S= E_{ij}^{m} E_{ik}^{n},\\
 T= E_{ji}^{p} E_{jk}^{q}.
\end{gather*}
Then $S^{-1} T$ factors as
\[ S^{-1} T = E_{ik}^{-n} E_{ij}^{-m} E_{ji}^{p} E_{jk}^{q} \go 
 \begin{cases} 
E_{ji}^{p} E_{jk}^{q+np} E_{ik}^{-n} & \text{if $m=0$}\\
E_{jk}^{q}E_{ik}^{-n-mq}  E_{ij}^{-m} & \text{if $p=0$}\\
           {\text\tt stop}                     & \text{if $m\neq0, p\neq 0$}                                 
\end{cases}
\]
\end{proof}

Now consider the case where $S$ and $T$ are both directed toward $i$:
\begin{gather*}
 S= E_{ij}^{m} E_{ik}^{n},\\
 T= E_{ij}^{p} E_{ik}^{q}.
\end{gather*}
To factor $S^{-1} T$, we can first use rule $(8)$ and rule $(2)$ to cancel elementary matrices with their inverses. Depending on the values of $m,n,p,q$, this brings us to the following $4$ cases:
\[ S^{-1} T \go
 \begin{cases}
  E_{ij}^{a} E_{ik}^{b}\\
  E_{ij}^{-a} E_{ik}^{-b}\\
  E_{ij}^{-a} E_{ik}^{b}\\
  E_{ik}^{-a} E_{ij}^{b}\\
 \end{cases}
\]
for some $a,b\geq 0$. In the first two cases there is nothing more to do. In the last two cases, if we only use rule $(6a)$, we can factor:
\begin{gather*}
 E_{ij}^{-a} E_{ik}^{b} \go E_{ik}^{b} E_{ij}^{-a}\\
 E_{ik}^{-a} E_{ij}^{b} \go E_{ij}^{b} E_{ik}^{-a}.
\end{gather*}

From this we get:

\begin{proposition} \label{prop-directedii}
 Let $S$ and $T$ be two sequence directed toward $i$. If we do not use rule $(6b)$ then $S^{-1} T$ has a unique factorization  $T S^{-1}$.
\end{proposition} \qed

\begin{corollary} 
When rule $(6b)$ is removed from the local algorithm $A$, then the algorithm is finite.
\end{corollary}

\begin{proof}
 It suffices to prove finiteness of $S^{-1} T$, where $S$ consists of one elementary matrix, or more generally, where $S$ is directed. Divide $T$ into directed sequences $T=T_1 T_2 \ldots T_n$. By previous propositions we know that 
\[  S^{-1} T_1 \go U V^{-1},\]
where both $U$ and $V$ are directed. By induction on $n$ the factorization of  $V^{-1} T_2 \ldots T_n$ is finite.
\end{proof}

To finish proving finiteness of the algorithm on $S^{-1} T$ where $S$ and $T$ are directed, the only case remaining is when 
\[ S^{-1} T = E_{ij}^{-m} E_{ik}^n\]
and we are allowed to use the full algorithm. 

\subsection{Factorization of $E_{ij}^{-m} E_{ik}^n$}

We will prove below that all factorizations of $E_{ij}^{-m} E_{ik}^n$ are finite. Since we are allowed to use both rules $(6a)$ and $(6b)$, there are in general many factorizations, and the number of different factorizations grows rapidly with $m$ and $n$. The table in Figure~\ref{table-number} lists the number of different factorizations (not counting the ones ending in {\text\tt stop}) for different values of $m=n$. These numbers were found using a computer.

\begin{figure}[h]
\begin{tabular}{c|c}
m=n & number of factorizations\\
\hline
1 & 2 \\ 
2 & 6 \\
3 & 16 \\
5 &  68\\
10 & 658 \\
20 & 8094 \\
30 & 37,322 \\
40 & 112,610 \\
\end{tabular}
\caption{Number of different factorizations of $E_{ij}^{-m} E_{ik}^n$.}
\label{table-number}
\end{figure}  

A {\em group} $H(j,k,i)$ is a sequence of elementary matrices of the form
\[ H(j,k,i)= (E_{jk}E_{ji})^{m_1} E_{kj}^{n_1} (E_{jk}E_{ji})^{m_2} E_{kj}^{n_2} \ldots (E_{jk}E_{ji})^{m_l} E_{kj}^{n_l} E_{jk} E_{ij}\]
for some $n_i, m_i, l \geq 0$. To have a unique expression for
$H(j,k,i)$ as above, we require that all $m_i, n_i >0$, except
possibly $m_1$ and $n_l$, and also that $l>0$. The shortest group is
simply  $H(j,k,i)=E_{jk} E_{ij}$. A {\em partial group} $H_p(j,k,i)$
is an initial segment in a group $H(j,k,i)$. We sometimes write
$H_p(j,k,i)_{\alpha\beta}$ to indicate that the partial group ends
with letter $E_{\alpha\beta}$. 

\begin{theorem} \label{thm-main}
 All factorizations of $E_{ij}^{-m} E_{ik}^n$ are finite and  if
\[ E_{ij}^{-m} E_{ik}^n \go T S^{-1},\]
then either $T=E_{ik}^n$ or $T$ has the form
\[ T= E_{ik}^q E_{ki} (H(j,k,i)R_{jki})^r H_p(j,k,i)\]
for some $q,r\geq 0$.
\end{theorem}

Note that since the algorithm is symmetric, $S$ in the statement of the theorem has the same form as $T$ (with indices $j$ and $k$ interchanged). The case $T=E_{ik}^n$ occurs when we commute $E_{ij}^{-m}$ with  $E_{ik}^n$ using rule $(6a)$. The other form of $T$ occurs when we apply rule $(6b)$ at least once.

Let us say that a sequence $T$ has the form $\diamond$ if it is as in
the theorem: 
\[ T= E_{ik}^q E_{ki} (H(j,k,i)R_{jki})^r H_p(j,k,i). \qquad (\diamond) \]
Given such a $T$, we write $T_{\alpha\beta}$ to indicate that the last
symbol of $T$ is $E_{\alpha\beta}$. 

We start with an auxiliary lemma.

\begin{lemma} \label{lem-aux}
Consider sequences
\begin{enumerate}
\item[(a)] $T_{kj} E^{-1}_{ij} (H(j,k,i)R_{jki})^r H_p(j,k,i),$
\item[(b)] $T_{jk} E^{-1}_{ik} R_{jki}(H(j,k,i)R_{jki})^r H_p(j,k,i),$
\end{enumerate}
where T is of the form $\diamond$. The algorithm is finite on both sequences
and produces factorizations $T_1 S^{-1}$, where $T_1$ again has the
form $\diamond$. 
\end{lemma}

\begin{proof} 
Note that both sequences have a single inverse
elementary matrix in them. We prove both parts of the lemma
simultaneously by induction on the number of elementary matrices to
the right of the inverse. 

Consider first the sequence $(a)$:
\[ T_{kj} E^{-1}_{ij} (H(j,k,i)R_{jki})^r H_p(j,k,i) = T_{kj}
E^{-1}_{ij} (E_{jk}E_{ji})^{m_1} E_{kj}^{n_1} \ldots.\]
If $m_1>0$, the factorization stops with $E^{-1}_{ij}E_{ji}$, so
assume $m_1=0$. If $n_1>0$, we apply one step of the algorithm:
\[ T_{kj} E^{-1}_{ij} E_{kj}^{n_1} \ldots \go  T_{kj}E_{kj}
E^{-1}_{ij} E_{kj}^{n_1-1} \ldots.\]
We can combine $T_{kj}E_{kj}$ into one $T'_{kj}$ that again has the
form $\diamond$ and we are back to the case of $(a)$, but  with one
less elementary matrix to the right of $E^{-1}_{ij}$. If also $n_1=0$,
then we apply the algorithm:
\begin{eqnarray*}
 T_{kj} E^{-1}_{ij} E_{jk}E_{ij} R_{jki}\ldots 
&\go& T_{kj} E_{jk} E^{-1}_{ik}E^{-1}_{ij} E_{ij} R_{jki} \ldots \\
&\go& T_{kj} E_{jk} E^{-1}_{ik} R_{jki} \ldots.
\end{eqnarray*}
We combine $T_{kj} E_{jk}$ into $T'_{jk}$, and this brings us inductively to the
case $(b)$. There is also the possibility that in the sequence $(a)$
the group $H(j,k,i)$ occurring is the last partial group, and in that
group either the last symbol $E_{ij}$ or both $E_{jk}E_{ij}$ are
missing. In both cases the algorithm terminates and the form of $T_1$
can be read off from the formulas above.

Now consider sequence $(b)$:
\[ T_{jk} E^{-1}_{ik} R_{jki}(H(j,k,i)R_{jki})^r H_p(j,k,i) \go
T_{jk} R_{jki} E^{-1}_{ji} (E_{jk}E_{ji})^{m_1} E_{kj}^{n_1}
\ldots. \qquad (*)\]
When $m_1>0$, we apply three steps of the algorithm to get:
\[T_{jk} R_{jki} E^{-1}_{ji} (E_{jk}E_{ji})^{m_1} E_{kj}^{n_1} \ldots
\go T_{jk} E_{ij} R_{jki} (E_{jk}E_{ji})^{m_1-1} E_{kj}^{n_1}
\ldots.\]
Note that $T_{jk} E_{ij}$ is of the form $\diamond$ and it ends with a
complete group. Thus:
\begin{eqnarray*}
T_{jk} E_{ij} R_{jki} &=& E_{ik}^q E_{ki} (H(j,k,i)R_{jki})^r,\\
(E_{jk}E_{ji})^{m_1-1} E_{kj}^{n_1}\ldots &=& (H(j,k,i)R_{jki})^s
H_p(j,k,i).
\end{eqnarray*} 
Concatenating these sequences gives the $T_1$ as stated in the lemma.

When $m_1=0$ and $n_1>0$, we apply the algorithm to the sequence $(*)$
as follows:
\begin{eqnarray*}
 T_{jk} R_{jki} E^{-1}_{ji} E_{kj}^{n_1}\ldots 
&\go& T_{jk}R_{jki}E_{kj}E_{ki}E^{-1}_{ji} E_{kj}^{n_1-1}\ldots \\
&\go& T_{jk}E_{ji}E_{jk}E^{-1}_{ik} R_{jki} E_{kj}^{n_1-1}\ldots.
\end{eqnarray*}
Combining $T_{jk}E_{ji}E_{jk} = T'_{jk}$, we are inductively back to
sequence $(b)$.

Finally, when $m_1=n_1=0$, we factor the sequence $(*)$ as
\begin{eqnarray*}
T_{jk} R_{jki} E^{-1}_{ji} E_{jk}E_{ij} R_{jki}\ldots
&\go&T_{jk}R_{jki}E_{kj}E_{ik}R_{kij}E^{-1}_{ki}E^{-1}_{ij} E_{ij} R_{jki}\ldots\\
&\go&T_{jk}R_{jki}E_{kj}E_{ik}R_{kij}E^{-1}_{ki} R_{jki}\ldots\\
&\go&T_{jk}E_{ji}E_{kj}E^{-1}_{ki}\ldots.
\end{eqnarray*}
This brings us by induction to the sequence $(a)$. We should also
consider the case where either $E_{ij}$ or both $E_{jk}E_{ij}$ are
missing from the final partial group. Both these cases are easy to deal
with and lead to the required form of $T_1$.
\end{proof}

\begin{proof}[Proof of the theorem]

To factor $E_{ij}^{-m} E_{ik}^n$, we use induction on $m$. If all
factorizations of $E_{ij}^{-(m-1)} E_{ik}^n$ have the claimed form $T
S^{-1}$, it suffices to prove that all factorizations of $E_{ij}^{-1}
T$ have the same form. The base case $m=0$ is trivial.

When $T=E_{ik}^n$, we can either commute $E_{ij}^{-1}$ with $T$ using
rule $(6a)$, or we can commute the first $p$ steps, then apply rule
$(6b)$:
\begin{eqnarray*}
E_{ij}^{-1} E_{ik}^n &\go& E_{ik}^p \underline{E_{ij}^{-1}E_{ik}} E_{ik}^q\\
&\go& E_{ik}^pE_{ki}E_{jk}R_{kji}E_{kj}^{-1}\underline{E_{ji}^{-1}E_{ik}^q}\\
&\go& E_{ik}^pE_{ki}E_{jk}R_{kji}\underline{E_{kj}^{-1}E_{ik}^q}
E_{jk}^{-q} E_{ji}^{-1}\\ 
&\go&  E_{ik}^pE_{ki}E_{jk}\underline{R_{kji}(E_{ik}E_{ij})^q}
E_{kj}^{-1}E_{jk}^{-q} E_{ji}^{-1}\\ 
&\go& E_{ik}^p E_{ki}E_{jk}(E_{ji}E_{jk})^q R_{kji}
E_{kj}^{-1}E_{jk}^{-q} E_{ji}^{-1}.
\end{eqnarray*}
Note that $E_{jk}(E_{ji}E_{jk})^q = (E_{jk}E_{ji})^q E_{jk}$ is a
partial group $H_p(j,k,i)$, thus the result
\[ T' =  E_{ik}^p E_{ki}H_p(j,k,i)\]
is as required.

Now let us assume that $T$ is of the form $\diamond$ and consider
factorizations of $E_{ij}^{-1}T$:
\[ E_{ij}^{-1}T =  E_{ij}^{-1} E_{ik}^n E_{ki} (H(j,k,i)R_{jki})^r
H_p(j,k,i).\]
From the above we know how to factor $E_{ij}^{-1} E_{ik}^n$. First
suppose that the factorization is $E_{ik}^nE_{ij}^{-1}$. Then we
continue with the algorithm:
\[ E_{ik}^nE_{ij}^{-1} E_{ki} (H(j,k,i)R_{jki})^rH_p(j,k,i) \go
E_{ik}^n E_{ki} E_{kj}E_{ij}^{-1}(H(j,k,i)R_{jki})^rH_p(j,k,i).\]
Since the initial segment $E_{ik}^n E_{ki} E_{kj} = T'_{kj}$, where
$T'$ is of the form $\diamond$, we are reduced to case $(a)$ of the
lemma.

Next suppose that we do not commute $E_{ij}^{-1}$ with all of
$E_{ik}^n$:
\begin{eqnarray*}
E_{ij}^{-1}T &\go& E_{ik}^p E_{ki}E_{jk}(E_{ji}E_{jk})^q R_{kji}
E_{kj}^{-1}E_{jk}^{-q} \underline{E_{ji}^{-1} E_{ki}} H(j,k,i) \ldots \\
&\go& E_{ik}^p E_{ki}E_{jk}(E_{ji}E_{jk})^q R_{kji}
E_{kj}^{-1}\underline{E_{jk}^{-q}E_{ki}} E_{ji}^{-1} H(j,k,i) \ldots \\
&\go& E_{ik}^p E_{ki}E_{jk}(E_{ji}E_{jk})^q R_{kji}
E_{kj}^{-1}E_{ki} (E_{ji}^{-1}E_{jk}^{-1})^q E_{ji}^{-1} H(j,k,i) \ldots \\
&=& E_{ik}^p E_{ki}(E_{jk}E_{ji})^q E_{jk} R_{kji}
E_{kj}^{-1}E_{ki} E_{ji}^{-1}(E_{jk}^{-1} E_{ji}^{-1})^q
(E_{jk}E_{ji})^{m_1} E_{kj}^{n_1}\ldots  \quad (**)
\end{eqnarray*}
At the next step we cancel pairs $E_{jk}E_{ji}$ with pairs $E_{jk}^{-1}
E_{ji}^{-1}$. The number of such cancellations depends on $q$ and $m_1$.

When $q\leq m_1-1$, we get:
\begin{eqnarray*}
 (**) &\go& E_{ik}^p E_{ki}(E_{jk}E_{ji})^qE_{jk} R_{kji}
\underline{E_{kj}^{-1}E_{ki}} E_{jk}(E_{jk}E_{ji})^{m_1-1-q} E_{kj}^{n_1}\ldots\\
&\go& E_{ik}^p E_{ki}(E_{jk}E_{ji})^q E_{jk} \underline{R_{kji}
E_{ik}E_{ji}}R_{ijk}E_{ij}^{-1} \underline{E_{jk}^{-1}
  E_{jk}}(E_{jk}E_{ji})^{m_1-1-q} E_{kj}^{n_1}\ldots \\
&\go& E_{ik}^p E_{ki}(E_{jk}E_{ji})^qE_{jk} 
E_{ji}E_{kj}E_{ij}^{-1}(E_{jk}E_{ji})^{m_1-1-q} E_{kj}^{n_1}\ldots 
\end{eqnarray*}
Here 
\[T'=E_{ik}^p E_{ki}(E_{jk}E_{ji})^qE_{jk} E_{ji}E_{kj} = E_{ik}^p
  E_{ki} (E_{jk}E_{ji})^{q+1}E_{kj} \]
is of the form $\diamond$ and we are in case $(a)$ of the lemma.

When $q=m_1$ and $n_1>0$, then we get
\begin{eqnarray*}
  (**) &\go& E_{ik}^p E_{ki}(E_{jk}E_{ji})^qE_{jk} R_{kji}
E_{kj}^{-1}E_{ki} \underline{E_{ji}^{-1} E_{kj}^{n_1}}\ldots\\
&\go& E_{ik}^p E_{ki}(E_{jk}E_{ji})^qE_{jk} R_{kji}
\underline{E_{kj}^{-1}E_{ki}} E_{kj}E_{ki} E_{ji}^{-1} E_{kj}^{n_1-1}\ldots\\
&\go& E_{ik}^p E_{ki}(E_{jk}E_{ji})^qE_{jk} R_{kji}
E_{ki} \underline{E_{kj}^{-1} E_{kj}} E_{ki} E_{ji}^{-1} E_{kj}^{n_1-1}\ldots\\
&\go& E_{ik}^p E_{ki}(E_{jk}E_{ji})^qE_{jk} \underline{R_{kji}
E_{ki}} E_{ki} E_{ji}^{-1} E_{kj}^{n_1-1}\ldots\\
&\go& E_{ik}^p E_{ki}(E_{jk}E_{ji})^qE_{jk} E_{ij} R_{kji} E_{ki}
    E_{ji}^{-1} E_{kj}^{n_1-1}\ldots.
\end{eqnarray*}
In the last sequence we use the fact that $R_{kji} = R_{jki}^2$ and
continue:
\begin{eqnarray*} & & E_{ik}^p E_{ki}(E_{jk}E_{ji})^qE_{jk} E_{ij} R_{jki} \underline{R_{jki} E_{ki}
    E_{ji}^{-1}} E_{kj}^{n_1-1}\ldots \\
&\go& E_{ik}^p E_{ki}(E_{jk}E_{ji})^qE_{jk} E_{ij} R_{jki} E_{jk}
    E_{ik}^{-1}R_{jki} E_{kj}^{n_1-1}\ldots.
\end{eqnarray*}
This expression falls into case $(b)$ of the lemma.

When $q=m_1$ and $n_1=0$, then we get
\begin{eqnarray*}
  (**) &\go& E_{ik}^p E_{ki}(E_{jk}E_{ji})^qE_{jk} R_{kji}
E_{kj}^{-1}E_{ki} \underline{E_{ji}^{-1} E_{jk}} E_{ij} \ldots\\
&\go& E_{ik}^p E_{ki}(E_{jk}E_{ji})^qE_{jk} R_{kji}
\underline{E_{kj}^{-1}E_{ki}  E_{kj}} E_{ik} R_{kij}E_{ki}^{-1}\underline{E_{ij}^{-1}E_{ij}} \ldots\\
&\go& E_{ik}^p E_{ki}(E_{jk}E_{ji})^qE_{jk} \underline{R_{kji}
E_{ki}} E_{ik} R_{kij}E_{ki}^{-1} \ldots\\
&\go& E_{ik}^p E_{ki}(E_{jk}E_{ji})^qE_{jk} E_{ij} R_{jki} \underline{R_{jki}
E_{ik} R_{kij} E_{ki}^{-1} R_{jki}}H(j,k,i) \ldots\\
&\go& E_{ik}^p E_{ki}(E_{jk}E_{ji})^qE_{jk} E_{ij} R_{jki}
E_{kj} E_{ij}^{-1}H(j,k,i) \ldots
\end{eqnarray*}
This is the sequence $(a)$ in the lemma. We also have to consider the
case where the final $E_{ij}$ or both $E_{jk} E_{ij}$ are missing, but
these cases are simple and left to the reader.

When $q>m_1$, we get
\[
  (**) \go
E_{ik}^p E_{ki}E_{jk}(E_{ji}E_{jk})^q R_{kji}
E_{kj}^{-1}E_{ki} (E_{ji}^{-1}E_{jk}^{-1})^{q-m_1} E_{ji}^{-1}
E_{kj}^{n_1} \ldots \]
When $n_1>0$, this sequence stops with $E_{jk}^{-1}E_{kj}$. When
$n_1=0$, the sequence 
\begin{eqnarray*}
  (**) &\go&
E_{ik}^p E_{ki}E_{jk}(E_{ji}E_{jk})^q R_{kji}
E_{kj}^{-1}E_{ki} (E_{ji}^{-1}E_{jk}^{-1})^{q-m_1} \underline{E_{ji}^{-1}
E_{jk}}E_{ij} \ldots \\
&\go&
E_{ik}^p E_{ki}E_{jk}(E_{ji}E_{jk})^q R_{kji}
E_{kj}^{-1}E_{ki} \underline{(E_{ji}^{-1} E_{jk}^{-1})^{q-m_1}
E_{kj}} E_{ik} R_{kij} E_{ki}^{-1} E_{ij}^{-1} E_{ij} \ldots 
\end{eqnarray*}
also stops with $E_{jk}^{-1}E_{kj}$. The cases where either the final
$E_{ij}$ or both $E_{jk} E_{ij}$ are missing are left to the reader.
\end{proof}

\subsection{A global finiteness result}

We prove here finiteness of the global algorithm $A$ in a special
case discussed in the introduction.

Consider two sequences of global star subdivisions of a single cone
$\langle v_i, v_j, v_k\rangle$. The subdivivision rays in one sequence
are generated by $v_i+v_j, 2v_i+v_j,\ldots, mv_i+v_j$, and in the other
sequence by $v_i+v_k, 2v_i+v_k,\ldots, nv_i+v_k$. To prove that the
algorithm is finite when applied to this sequence of $m$ star
assemblies followed by $n$ star subdivisions, we follow the notation
in the proof of Proposition~\ref{prop-local-global}. To prove finiteness of the
global algorithm, it suffices to prove finiteness of the local
algorithm when applied to all local subsequences of the original
global sequence. The local subsequences of the $m$ star assemblies -
$n$ star subdivisions are:
\begin{gather*}
E_{ij}^{-m} E_{ik}^n, \\
E_{ij}^{-m} E_{ik}^q E_{ki}, \\
E_{ji}^{-1}E_{ij}^{-p} E_{ik}^n,\\
E_{ji}^{-1}E_{ij}^{-p}  E_{ik}^q E_{ki},
\end{gather*}
where $0\leq p < m$ and $0\leq q<n$. Finiteness of the local algorithm
when applied to the first sequence was proved in the previous
subsection. The proof of Thorem~\ref{thm-main} also covers the case of
the second sequence because $T=E_{ik}^q E_{ki}$ is of the form
$\diamond$ and we proved that for any such $T$, the factorizations of
$E_{ij}^{-m} T$ are finite. The case of the third sequence follows by
symmetry. Only the last case is remaining. Using the proof of the
theorem, we can factor
\[ E_{ji}^{-1}E_{ij}^{-p}  E_{ik}^q E_{ki} \go E_{ji}^{-1} T S^{-1},\]
where $T$ is of the form $\diamond$. Thus it suffices to prove that
all factorizations of $E_{ji}^{-1} T$ are finite. The resulting
factorizations may not follow the same pattern as in the theorem.

We apply the algorithm as follows:
\begin{eqnarray*}
E_{ji}^{-1} T &=& \underline{E_{ji}^{-1} E_{ik}^q} E_{ki}
H(j,k,i) \ldots \\
&\go& E_{ik}^q E_{jk}^{-q}\underline{E_{ji}^{-1}E_{ki}}H(j,k,i)\ldots \\
&\go& E_{ik}^q \underline{E_{jk}^{-q}E_{ki}}E_{ji}^{-1}H(j,k,i)\ldots \\
&\go& E_{ik}^q E_{ki}
(E_{ji}^{-1}E_{jk}^{-1})^qE_{ji}^{-1}(E_{jk}E_{ji})^{m_1} E_{kj}^{n_1}
\ldots .
\end{eqnarray*}
The part of the sequence that needs to be factored:
\[ (E_{ji}^{-1}E_{jk}^{-1})^qE_{ji}^{-1}(E_{jk}E_{ji})^{m_1}
E_{kj}^{n_1}\ldots =E_{ji}^{-1}(E_{jk}^{-1}E_{ji}^{-1})^q(E_{jk}E_{ji})^{m_1}
E_{kj}^{n_1}\ldots  
\]
also appears as a part of the sequence $(**)$ in the proof of
Thorem~\ref{thm-main}. We know that all factorizations of $(**)$ are
finite. This does not directly imply that all factorizations of the
current sequence are finite. One can, however, repeat the proof of
finiteness of the algorithm on $(**)$, adjusting the steps where
necessary for the current sequence.
 
To finish the discussion of the global case, note that in the common
refinement of the two sequences of global star subdivisions, the
number of maximal cones is equal to the number of different
factorizations of the four types of local sequences above. From
Figure~\ref{table-number} we know that this number is very large when
$m$ and $n$ are large. Since each star subdivision of a fan increases
its number of maximal cones by one or two, it also follows that the
number of star subdivisions and star assemblies in the strong
factorization of the original $m$-by-$n$ sequence is very large.

\section{Further directions}

To prove finiteness of the local algorithm $A$, one needs to find a
way to bound the complexity of the factorizations, for example their
length. The proofs of finiteness presented here do not construct such a
bound. Instead, we considered very regular initial sequences and
proved that the factorizations then are also regular. The proof of
Theorem~\ref{thm-main} for example does not give a bound on the length
of factorizations of $E_{ij}^{-m} E_{ik}^n$. In computer experiments
one can see that the factorizations in these cases are in fact rather
short. If a 
factorization has the form $T S^{-1}$, then the number of elementary
matrices in $T$ is no more than $\max\{2m+n, m+2n\}$. We discuss in
this section two approaches that may lead to such bounds on complexity
and to a proof of finiteness.

\subsection{The Cayley graph}

It is well-known that the elementary matrices $E_{ij}$ generate the
group $SL(3,\ZZ)$. Construct the Cayley graph of $SL(3,\ZZ)$ using
these generators. This graph has its vertices as the matrices in
$SL(3,\ZZ)$ and there is an edge from $X$ to $Y$ if $Y=X E_{ij}$ for
some $i,j$. If we represent cones by matrices, then a sequence of
local star subdivisions is simply a path in this graph. To make the
representation independent of the chosen order of generators for a
cone, we should really consider the quotient of the Cayley graph by
the alternating group $A_3$ that acts by cyclically permutating
columns of matrices. Then all factorization diagrams, such as the
example in Figure~\ref{fig-local-example}, can be thought of as being
subgraphs of the quotient graph.

Cayley graphs of groups such as $SL(3,\ZZ)$ are studied in
combinatorial group theory. There are many techniques and results
known that are similar to our problem. For example, in studying the
isoperimetric 
inequalities for the Cayley graph in the word metric, one starts with a 
closed loop and asks how many relations are needed to contract this
loop to a point. This can be compared to the factorization problem as
follows: if we start with a partial factorization diagram, which is a
loop in the graph, we seek to expand this loop by applying the
relations. It may be possible to apply techniques from combinatorial
group theory to get invariants for the factorization problem.

\subsection{Factorization along a valuation}

Define a valuation as a ray in $\RR^3$ generated by a vector $v$ with
rationally independent coordinates. In the local algorithm the
valuation ray tells us which cone to choose after a star subdivision
-- we always choose the cone containing the ray. Using the local
algorithm when the choice of cones is given by a valuation ray is
called factorization along a valuation. 

One can pose the following conjecture:

\begin{conjecture} Algorithm $A$ is finite along any valuation
\end{conjecture}

As in the local case, we also have the following problem: does
finiteness of the algorithm along any valuation imply finitness of the
local or global algorithm? 

The factorization algorithm along a valuation is easier to visualize
than the global or the local algorithm. If a cone
$\langle v_1, v_2, v_3\rangle$ contains the vector $v$, write 
\[ v = b_1 v_1 + b_2 v_2 + b_3 v_3\]
for some positive numbers $b_1, b_2, b_3$. We can then represent the
cone by the vector $(b_1,b_2,b_3)$ in $\RR^3$. A local star subdivision of
this cone along the valuation corresponds to
subtracting $b_i$ from $b_j$ for some $i\neq j$, for example
\[ (b_1, b_2-b_1, b_3) \to (b_1,b_2,b_3)\]
is one star subdivision, provided that $b_2>b_1$. Using this
representation of cones as points in $\RR^3$, one can also consider
embedding a factorization diagram in $\RR^3$. As before, to make this
independent of the order of generators, the triple $(b_1,b_2,b_3)$
should be considered up to cyclic permutation. Because of this cyclic
permutation, the ambient space $\RR^3/A_3$ becomes more complicated to
work with. In $\RR^3$ or in the quotient $\RR^3/A_3$ where 
one can measure lengths and distances, it may be possible to
find numerical invariants that bound the complexity of the algorithm.

We also remark that the local algorithm $B$ does not need the cyclic
permutations, hence a factorization diagram can be embedded in
$\RR^3$. However, it is not hard to find a valuation and a diagram
such that the factorization algorithm $B$ along the valuation is not
finite. (The example in Figure~\ref{fig-example}(b) is actually finite
along any valuation, but if one takes a symmetric $2$-by-$2$ initial
sequence, then there are many valuations along which the algorithm is
not finite.)

\end{document}